\documentclass[a4paper,10pt]{amsart}

\usepackage{enumerate}
\usepackage[all, knot]{xy}
\usepackage{amsthm,amsmath,amssymb,amsfonts}
\usepackage{subcaption}
\usepackage{array}
\usepackage[disable]{todonotes} 
\usepackage[pdftex, pdfborder= 0 0 0, citecolor=black, urlcolor=black,
  linkcolor=black, colorlinks=true, bookmarksopen=true]{hyperref}
\newtheorem{thm}{Theorem}[section]
\newtheorem{lemma}[thm]{Lemma}
\newtheorem{cor}[thm]{Corollary}
\newtheorem{prop}[thm]{Proposition}

\theoremstyle{definition}
\newtheorem{defi}[thm]{Definition}

\newtheorem{rmk}[thm]{Remark}
\newtheorem{exm}[thm]{Example}
\newcommand{\Enr}[1]{#1 \text{-}}
\newcommand{\Cat}{\mathsf{Cat}}
\newcommand{\Catadj}{\mathsf{Cat}_{\mathrm{adj}}}
\newcommand{\ECat}[1]{#1 \text{-}\Cat}
\newcommand{\ECatfc}[2]{\ECat{#1}_{#2}}
\newcommand{\Oper}{\mathsf{Oper}}
\newcommand{\EOper}[1]{#1 \text{-} \Oper}
\newcommand{\EOperfc}[2]{\EOper{#1}_{#2}}
\newcommand{\Cl}[1]{\mathbf{Col}(#1)}
\newcommand{\Opop}[1]{\mathsf{Op}_{#1}}
\newcommand{\Opcat}[1]{\mathsf{Cat}_{#1}}
\newcommand{\EMGraph}[1]{#1 \text{-} \mathsf{MultiGraph}}
\newcommand{\EGraph}[1]{#1 \text{-} \mathsf{Graph}}
\newcommand{\NSOper}{\mathsf{NSOper}}
\newcommand{\ENSOper}[1]{#1 \text{-}\NSOper}
\newcommand{\ENSOperfc}[2]{\ENSOper{#1}_{#2}}

\newcommand{\N}{\mathbb{N}}
\newcommand{\id}{\mathrm{id}}
\newcommand{\Set}{\mathcal{S}et}
\newcommand{\Top}{\mathsf{Top}}

\newcommand{\singSet}{\{*\}}
\newcommand{\HoC}[1]{\mathrm{Ho}(#1)}
\newcommand{\cell}[1]{#1 \mathrm{-cell}}
\newcommand{\inj}[1]{#1 \mathrm{-inj}}
\newcommand{\cof}[1]{#1 \mathrm{-cof}}

\newcommand{\Mm}{\mathcal{V}}
\newcommand{\Un}{\mathcal{I}}
\newcommand{\arr}[2]{#1 \longrightarrow #2}
\newcommand{\arro}[3]{#2 \!\overset{#1}\longrightarrow \! #3}

\newcommand{\morp}[3]{#1 \colon #2 \longrightarrow ~#3}
\newcommand{\ntra}[3]{#1 \colon #2 \Longrightarrow #3}
\newcommand{\gfun}[5]{
\begin{array}{r@{}c@{\hspace{0.1cm}}c@{\hspace{0.1cm}}l}
#1 \colon & #2 & \longrightarrow & #3 \\
             & #4 & \longmapsto & #5\\
\end{array}
}
\newcommand{\indl}[1]{{#1}_!}
\newcommand{\indr}[1]{{#1}^*}
\newcommand{\adjpair}[4]{#1 \colon #3 \rightleftarrows #4 \colon #2}
\newcommand{\indadjpair}[3]{\adjpair{{#1}_!}{{#1}^*}{#2}{#3}}
\newcommand{\diagc}[9]{\xymatrix{#1 \ar @{} [dr] |{#9} \ar[d]_{#5} \ar[r]^-{#6} &  #2 \ar[d]^-{#7}\\
          #3 \ar[r]_-{#8} & #4}}
\newcommand{\coeq}[4]{\xymatrix{#1 \ar@<0.7ex>[r]^-{#3} \ar@<-0.7ex>[r]_-{#4} &  #2}}
\newcommand{\coeqv}[4]{\xymatrix{#1 \ar@<0.7ex>[d]^-{#3} \ar@<-0.7ex>[d]_-{#4} \\  #2}}

\newcommand{\trees}[1]{\mathbf{T}(#1)}

\newcommand{\edge}[1]{\mathrm{edge}(#1)} 
\newcommand{\verx}[1]{\mathrm{vert}(#1)}
\newcommand{\leav}[1]{\mathrm{in}(#1)} 
\newcommand{\card}[1]{|#1|}
\newcommand{\arity}[1]{a(#1)}

\newcommand{\Opnsop}[1]{\mathsf{NSOp}_{#1}}
\newcommand{\troot}[1]{\mathrm{r}(#1)}

\newcommand{\opc}[1]{#1^{\mathrm{op}}}

\newcommand{\Sqofc}[1]{\mathrm{Seq}(#1)}
\newcommand{\RSqofc}[1]{\mathrm{RSeq}(#1)}
\newcommand{\treesc}[1]{\mathbf{T}_{#1}}

\newcommand{\mf}[1]{M_{#1}}

\newcommand{\LaxSCat}[1]{\Cat// #1}

\newcommand{\homcc}[1]{\pi_0(#1)}
\newcommand{\homcat}[1]{\mathrm{Ho}(#1)}
\newcommand{\EROper}[1]{#1 \text{-} \mathsf{ROper}}
\newcommand{\EROperfc}[2]{\EROper{#1}_{#2}}
\newcommand{\ERMGraph}[1]{#1 \text{-} \mathsf{RMultiGraph}}
\newcommand{\Oprop}[1]{\mathsf{ROp}_{#1}}

\newcommand{\Obf}{\mathbf{Ob}}
\newcommand{\Clf}{\mathbf{Cl}}
\newcommand{\trf}[1]{#1^{u}}
\newcommand{\trc}[1]{#1_{u}}
\newcommand{\Fibf}[1]{\mathcal{F}\mathrm{ib}_{#1}}
\newcommand{\Fib}[2]{\Fibf{#1}(#2)}
\newcommand{\Alg}[2]{\mathrm{Alg}_{#1}(#2)}
\newcommand{\obCat}{\mathbf{1}}

\numberwithin{equation}{subsection}

\title{A Model Structure for Enriched Coloured Operads}
\author{Giovanni Caviglia}

\begin{document}

\setcounter{tocdepth}{1}
\begin{abstract}
We prove that, under certain conditions, the model structure on a monoidal model category $\Mm$ can be transferred to a model structure on the category of $\Mm$-enriched coloured (symmetric) operads. As a particular case we recover the known model structure on simplicial operads.  
\end{abstract}

\maketitle

\tableofcontents

\section{Introduction}

Le $\Mm$ be a cofibrantly generated monoidal model category. In their paper \cite{BM12}, Berger and Moerdijk proved that under suitable assumptions on $\Mm$  the canonical model structure exists on $\ECat{\Mm}$, the category of (small) $\Mm$-enriched categories.

Symmetric coloured operads (also called symmetric multi-categories) can be regarded as generalized categories in which we allow the morphisms, now called \emph{operations}, to have an arbitrary number of inputs. 

Our goal in this paper is to generalize the results of \cite{BM12} in order to get a canonical model structure on $\EOper{\Mm}$, the category of $\Mm$-enriched coloured operads.

By the canonical model structure we mean the unique one (when it exists) for which the fibrant objects are the locally fibrant ones and the trivial fibrations are the local trivial fibrations which are surjective on the colours.

The main result of the paper is Theorem \ref{mod.str.thm} which gives sufficient conditions on a right proper model category $\Mm$ in order to guarantee the existence of the canonical model structure on $\EOper{\Mm}$. For $\Mm$ equal to simplicial sets one recovers the model structure on simplicial coloured operads presented in \cite{CM11} and \cite{Ro11}.

The existence of the canonical model structure on $\ENSOper{\Mm}$ and $\EROper{\Mm}$ (the categories of non-symmetric coloured operads and symmetric reduced operads respectively) is also established in Theorem $\ref{mod.str.thm}$ under weaker assumptions on $\Mm$.

It turns out that our hypotheses on $\Mm$ have to be stricter than the ones made in \cite{BM12} for a series of reasons. Given a set $C$ let us denote by $\ECatfc{\Mm}{C}$ the category whose objects are the $\Mm$-enriched categories with $C$ as set of objects and whose morphisms are the functors between them which are the identity at the level of objects. The category $\EOperfc{\Mm}{C}$ can be described as the category of algebras for a certain non-symmetric coloured operad $\Opcat{C}$ (whose set of colours is $C\times C$); if $\ECat{\Mm}$ admits the canonical model structure then $\ECatfc{\Mm}{C}$ admits the model structure transferred from $\Mm^{C\times C}$ through the free-forgetful adjunction induced by $\Opcat{C}$. In other words $\Opcat{C}$ has to be admissible in $\Mm$ for every $C\in \Set$.

In \cite{BM05} and \cite{Mu11} sufficient conditions on $\Mm$ are given in order to guarantee that every non-symmetric operad is admissible in $\Mm$.

The case of $\EOper{\Mm}$ is somehow similar: for every set $C$ we can consider the subcategory $\EOperfc{\Mm}{C}$ whose objects are the operads with $C$ as set of colours and whose morphisms are the ones which are the identity on colours. This time $\EOperfc{\Mm}{C}$ is the category of algebras of a symmetric operad $\Opop{C}$, that does not come from a non-symmetric one. 
For the canonical model structure over $\EOper{\Mm}$ to exists, we need $\Opop{C}$ to be admissible in $\Mm$ for every $C\in \Set$, and this can not be guaranteed  by the hypotheses in \cite{BM12}.

The fact that the operad is not non-symmetric forces us to restrict our hypotheses on $\Mm$ in order to apply \cite[Theorem 2.1]{BM05}.

Something more can be said in the case in which we restrict our-self to the category of non-symmetric coloured operads or to the category of reduced (symmetric) coloured operads (i.e. operads with no operations of arity $0$).

Again for every $C\in \Set$ we have an operad $\Opnsop{C}$ (resp. $\Oprop{C}$) whose algebras are non-symmetric (resp. reduced) coloured operads with set of colours $C$.
The operads $\Opnsop{C}$ are $\Oprop{C}$ are not non-symmetric but there is still some control on the actions of the symmetric groups since they give rise to tame polynomial monads. 

Admissibility of tame polynomial monads was studied by Batanin and Berger in \cite{BB13}.
Their results allow us to enlarge the class of model categories $\Mm$ for which $\ENSOper{\Mm}$ (the category of $\Mm$-enriched non-symmetric operads) and $\EROper{\Mm}$  (the category of $\Mm$-enriched reduced operads) admit the canonical model structure.


In the last section it is proved that under our assumptions the canonical model structure is right proper as well.

\subsection*{Acknowledgements} I am grateful to Ieke Moerdijk for having suggested the problem and for several suggestions concerning this work. I would also like to thank Dimitri Ara, Moritz Groth and Javier J. Gutiérrez for many helpful discussions. 

\section{Preliminaries}
Let $(\Mm,\otimes,I)$ be a symmetric monoidal category.

In this section we briefly recall the definitions of coloured operad enriched in $\Mm$ and of the category of algebras associated to it.

\subsection{Coloured operads}\label{col.op}

For every $n\in \N$ let $[n]\in \Set$ be the set of all positive natural numbers smaller or equal to $n$;

Given a set $C$, the set of signatures in $C$ is defined as 
\[
\Sqofc{C}=\{(c_1,\dots,c_n;c)\ | \ n\in \N \ \text{ and } c,c_1,\dots,c_n\in C \}=(\underset{n\in N}\coprod C^n)\times C\, .
\] 
For every element $s=(\bar{s},r)\in \Sqofc{C}$, the sequence $\bar{s}$ will be denoted $\leav{s}$ and $r$ will be indicated $\troot{s}$; there is a unique $n\in \N$ such that $\leav{s}\in C^n$, this number is called the \emph{valence} of  $s$ and denoted by $\card{s}$.

We will denote by $\RSqofc{C}$ the subset of $\Sqofc{C}$ formed by the signatures with valence different from $0$.

For every $s=(s_1,\dots,s_n; r)\in \Sqofc{C}$ a sequence of signatures $(t_1,\dots,t_n)$ is said to be \emph{composable} with $s$ if $\troot{t_i}=s_i$ for every $i\in [n]$; if this is the case we can produce the signature $s\circ (t_1,\dots,t_n)$. This new signature is such that $\troot{s\circ (t_1,\dots,t_n)}=r$, the valence is $\sum_{i\in [n]}\card{t_i}$ and $\leav{s\circ (t_1,\dots,t_n)}=(\leav{t_1},\dots,\leav{t_n})$.

Every map of sets $\morp{f}{C}{D}$ induces a map 
\[
\gfun{f}{\Sqofc{C}}{\Sqofc{D}}{(s_1,\dots,s_n;s)}{(f(s_1),\dots,f(s_n);f(s))}
\]

A symmetric $C$-coloured operad enriched in $\Mm$
is an object $\mathcal{O}$ in the product category $\Mm^{\Sqofc{C}}$ with the following additional data:
\begin{itemize}
\item
For every $s=(c_1,\dots,c_n; c) \in \Sqofc{C}$ and every sequence of signatures $(t_1,\dots,t_n)$ composable with $s$ a morphism
\[
\morp{\Gamma}{(\underset{i\in [n]}\bigotimes \mathcal{O}(t_i))\otimes \mathcal{O}(s)}{\mathcal{O}(s\circ(t_1,\dots,t_{n}))}
\]
called \emph{composition morphism}.
\item For every $s\in \Sqofc{C}$ and every $\sigma\in \Sigma_{\card{s}}$ a morphism
\begin{equation}\label{eq.perm}
\morp{\sigma^*}{\mathcal{O}(s)}{\mathcal{O}(s\sigma)}
\end{equation}
called \emph{symmetric morphism} (here $s\sigma$ stands for $(c_{\sigma(1)},\dots,c_{\sigma(n)};c)$).
\item For every $c\in C$ a morphism
\[
\morp{u_c}{I}{\mathcal{O}(c;c)}
\]
called the \emph{identity operation} of $c$.
\end{itemize}
These morphisms are required to satisfy some associativity, unitality and equivariance conditions.

If we drop the actions of the permutation groups (\ref{eq.perm})  (and hence the equivariance conditions) from the definition, we get non-symmetric coloured operads (sometimes called multi-categories).
If we replace $\Sqofc{C}$ by $\RSqofc{C}$, i.e. we do not allowed the operad to have operations with empty input, then we get the definition of reduced operad in the sense of \cite{BB13} (also called non-unitary operad in the literature).

A morphism of $C$-coloured operads is, informally speaking, a morphism in $\Mm^{\Sqofc{C}}$ respecting composition, identity operations and symmetric morphisms. The category of $C$-coloured operads enriched in $\Mm$ will be indicated by $\EOperfc{\Mm}{C}$. 
In the same way we can define $\ENSOperfc{\Mm}{C}$ and $\EROperfc{\Mm}{C}$, the categories of non-symmetric $C$ coloured operads and reduced $C$-coloured operads. 

For more details about the definitions of coloured operads and their morphisms see for example \cite{EM06} or \cite{BM05}.

It is also possible to define morphisms between operads with different sets of colours, we are now going to recall how.
For every map of sets $\morp{f}{C}{D}$ there is a ``forgetful'' functor
\begin{equation}\label{fstardef}
\morp{f^*}{\EOperfc{\Mm}{D}}{\EOperfc{\Mm}{C}}.
\end{equation}
For every $D$-coloured operad $\mathcal{O}$ and every $s\in\Sqofc{C}$ the $C$-coloured operad $f^*\mathcal{O}$ is defined on the $s$-component as $f^*\mathcal{O}(s)=\mathcal{O}(f(s))$.
Composition, identity operations and symmetries are defined in the obvious way using the ones in $\mathcal{O}$.

A morphism between a $C$-coloured operad $\mathcal{O}$ and a $D$-coloured operad $\mathcal{P}$ is defined as a couple $(f,\hat{f})$ where $f$ is a map from $C$ to $D$ and $\hat{f}$ is a morphism of $C$-coloured operads $\morp{\hat{f}}{\mathcal{O}}{f^*\mathcal{P}}$. Given an $E$-coloured operad $\mathcal{Q}$ and a morphism $(g,\hat{g})$ from $\mathcal{P}$ to $\mathcal{Q}$ we can define the composition $(g,\hat{g})\circ (f,\hat{f})$ as $(gf,\hat{f}f^*(\hat{g}))$.

In this way we can define the category $\EOper{\Mm}$ of coloured operads (without a fixed set of colours). Given a coloured operad $\mathcal{O}$ its set of colours will be denoted by $\Cl{\mathcal{O}}$.

Given two monoidal model category $\mathcal{C}$, $\mathcal{D}$ and a monoidal functor between them $\morp{F}{\mathcal{C}}{\mathcal{D}}$ it is easy to see that it passes to the categories of operads, i.e. there is a functor $\morp{F_{\sharp}}{\EOper{\mathcal{C}}}{\EOper{\mathcal{D}}}$ such that for every $\mathcal{O}\in\EOper{\mathcal{C}}$ one has $\Cl{F_{\sharp}(\mathcal{O})}=\Cl{\mathcal{O}}$ and $F_{\sharp}\mathcal{O}(s)=F(\mathcal{O}(s))$ for every $s\in \Sqofc{\Cl{\mathcal{O}}}$.

In particular if $(\Mm,\otimes,I)$ is a closed monoidal category there is always a strong monoidal functor
\begin{equation}
 \gfun{o^\Mm}{(\Set,\times,\ast)}{(\Mm,\otimes,I)}{S}{\underset{S}\coprod\, I}
\end{equation}
thus every $\Set$-operad $\mathcal{O}$ admits an incarnation $o^\Mm_\sharp(\mathcal{O})$ in $\Mm$.

\subsection{$\mathcal{O}$-Algebras}\label{sec.alg}
Assume that $\Mm$ is a bicomplete closed symmetric monoidal category. We briefly recall the definition of algebra for a coloured operad $\mathcal{O}$. Let $C=\Cl{\mathcal{O}}$, then an $\mathcal{O}$-algebra in $\Mm$ is an object of the product category $\Mm^{C}$ together with a morphism
\begin{equation}\label{eq.action}
 \morp{\alpha_s}{\mathcal{O}(s)\otimes A(c_1)\otimes\dots\otimes A(c_n)}{A(c)}
\end{equation}
for every $s=(c_1,\dots,c_n;c)\in \Sqofc{C}$. These \emph{action maps} have to satisfy some obvious conditions of associativity, unitality and equivariance (\emph{cf.} \cite{EM06}, \cite{BM05}). A morphism of $\mathcal{O}$-algebras is a morphism in $\Mm^{C}$ which respects the action maps.

Let $\Alg{\mathcal{O}}{\Mm}$ be the category of $\mathcal{O}$-algebra in $\Mm$. There is a (finitary) monadic adjunction:
\begin{equation}\label{adj.alg}
 \adjpair{F_{\mathcal{O}}}{U_{\mathcal{O}}}{\Mm^{C}}{\Alg{\mathcal{O}}{\Mm}}
\end{equation}
If $\mathcal{Q}$ is another operad with set of colours $D$ and $\morp{f}{\mathcal{O}}{\mathcal{Q}}$ is a morphism of operads then we have an adjunction:
\[
 \indadjpair{f}{\Alg{\mathcal{O}}{\Mm}}{\Alg{\mathcal{Q}}{\Mm}}
\]
and a commutative diagram of adjunctions:
\begin{equation}\label{eq.adj.op}
 \xymatrix{\Alg{\mathcal{O}}{\Mm} \ar@<1ex>[r]^{f_!} \ar@<1ex>[d]^{U_{\mathcal{O}}} & \Alg{\mathcal{Q}}{\Mm} \ar@<1ex>[l]^{f^*} \ar@<-1ex>[d]_{U_{\mathcal{Q}}}\\ 
                                \Mm^{C} \ar@<-1ex>[r]_{f_!} \ar@<1ex>[u]^{F_{\mathcal{O}}}& \Mm^{D} \ar@<-1ex>[l]_{f^{*}} \ar@<-1ex>[u]_{F_{\mathcal{Q}}}}
\end{equation}
where the adjunction on the lower row is given by the inverse image functor $f^*$ of the map $\morp{f}{C}{D}$ induced on the colours and its left adjoint.  

If $\mathcal{O}$ is a $\Set$-operad we will denote by $\Alg{\mathcal{O}}{\Mm}$ the category of algebras of $o^\Mm_\sharp(\mathcal{O})$.
 
\subsection{Admissible monads}
Consider $\Mm$, a cofibrantly generated category with set of generating (trivial) cofibrations $I$ (resp. $J$). Suppose that a finitary monad $T$ over $\Mm$ is given. This amount of data produces a ``free-forgetful'' adjunction
\[
 \adjpair{F_T}{U_T}{\Mm}{\Alg{T}{\Mm}}
\]
between $\Mm$ and the category of $T$-algebras. The category of $T$-algebras is complete and cocomplete and one can wonder if the model structure on $\Mm$ can be transferred along $F_T$, i.e. if there is a (necessarily unique)
cofibrantly generated model structure on $\Alg{T}{\Mm}$ for which a morphism of $T$-algebras $f$ is a weak equivalence (a fibration) if and only if $U_T(f)$ is a weak equivalence (resp. a fibration)
in $\Mm$ and whose set of generating (trivial) cofibrations is $F_T(I)$ (resp. $F_T(J)$). 
If such a model structure exists $(F_T,U_T)$ is a Quillen adjunction and $T$ is said to be \emph{admissible}.

We recall that a class of morphisms $S$ in a category $\Mm$ is \emph{saturated} if and only if it is closed under push-outs, retracts and transfinite composition. A class of morphisms $S$ in $\Mm$ which is closed under push-outs and transfinite composition is called \emph{weakly saturated}. Given a set
of morphisms $K$ in $\Mm$ a \emph{relative $K$-cell} is a morphism in the smallest weakly saturated class containing $K$.  

Let us fix a saturated class of morphisms $K$ in $\Mm$. 

\begin{defi}\label{Kadm}
A finitary monad $T$ on $\Mm$ is said to be \emph{$K$-admissible} if for every $\morp{i}{A}{B}$ in $I\cup J$
and every morphism of algebras $\morp{\alpha}{F_T(A)}{R}$, the push-out
\[
 \diagc{F_T(A)}{R}{F_T(B)}{R[i,\alpha]}{F_T(i)}{\alpha}{i_\alpha}{}{}
\]
 yields a morphism $i_\alpha$ such that:
\begin{enumerate}[i)]
 \item\label{item.adm1} $U_T(I_\alpha)$ belongs to $K$;
 \item\label{item.adm2} if $i\in J$ then $U_T(I_\alpha)$ is a weak equivalence. 
\end{enumerate}  
\end{defi}

We recall that a \emph{path object for a $T$-algebra $X$} is a factorization (in $\Alg{T}{\Mm}$) of the diagonal map
\[
 \xymatrix{X \ar[dr]_{p}\ar[rr]^{(\id,\id)} &    & X\times X \\
           & P(X)\ar[ur]_{q} &}
\]
such that $U_T(p)$ is a fibration and $U_T(q)$ is a weak equivalence.
A \emph{fibrant replacement} functor for $\Alg{T}{\Mm}$ is an endofunctor $\morp{R}{\Alg{T}{\Mm}}{\Alg{T}{\Mm}}$ together with a natural tranformation $\ntra{r}{\id}{R}$
such that $U_T(R(X))$ is fibrant for every $X\in \Alg{T}{\Mm}$ and $U_T r$ is a natural weak equivalence.
 
\begin{defi}
 A monad $T$ on $M$ \emph{has path objects} if the following two conditions are satisfied:
\begin{itemize}
 \item $\Alg{T}{\Mm}$ has a fibrant replacement functor;
 \item Every fibrant $T$-algebra admits a functorial path object.
\end{itemize}
\end{defi}

\begin{prop}\label{path.admis}
 If a finitary monad $T$ on $\Mm$ has path objects then all relative $F_T(J)$-cells are weak equivalences.
 In particular a finitary monad with path objects
satisfying condition \ref{item.adm1}) of Definition \ref{Kadm} is $K$-admissible. 
\end{prop}

For a proof of this proposition we refer the reader to \cite[Section 2.6]{BM03} and the references there.

\begin{defi}
Let $K$ be a class of morphisms in $\Mm$, a cofibrantly generated model category. $\Mm$ is said to be \emph{$K$-compactly generated} if:
\begin{itemize}
 \item every object is small respect to $K$;
 \item The class of weak equivalences is $K$-perfect (i.e. weak equivalences are closed under filtered colimits along morphisms of $K$).
\end{itemize}
If $K$ is the whole class of morphisms of $\Mm$ we say that $\Mm$ is \emph{strongly compactly generated}. 
\end{defi}

Examples of strongly cofibrantly generated category are: the category of Simplicial sets with the Kan-Quillen model structure, the category of Chain Complexes over a commutative ring with the projective model structure, the category of Simplicial Modules over a commutative ring.

The category of compactly generated (weak Hausdorff) Topogical Spaces is $\mathcal{T}_1$-compactly generated (\emph{cf.} \cite{BM12}), where $\mathcal{T}_1$ is the class of $T_1$ closed inclusions (see Appendix \ref{App.C} for the definition).  

\begin{prop}{\cite[Theorem 2.11]{BB13}}\label{admis.mon}
 Every finitary $K$-admissible monad on a $K$-compactly generated category $\Mm$ is admissible.
\end{prop}

\subsection{Admissible operads}\label{sec.adm.op}
Suppose that $\Mm$ is a cofibrantly generated monoidal model category with set of generating (trivial) cofibrations $I$ (resp. $J$). For every set $C$ the product category $\Mm^{C}$ inherits a model structure in which the weak equivalences (cofibrations, fibrations) are the level-wise ones (\emph{cf.} \cite{Ho99}).

For every $c\in C$ let $\morp{p_c}{\Mm^{C}}{\Mm}$ be the projection on the $c$-component. The functor $p_c$ has a left adjoint $\iota_c$ that is defined on the objects in the following way: for every $A\in \Mm$ and $t\in C$
\[
 \iota_c(A)(t)=
 \begin{cases}
  A & \text{if } c=t\\
 \emptyset & \text{otherwise}
 \end{cases}
\]
where $\emptyset$ is the initial object of $\Mm$.

The model structure on $\Mm^{C}$ is cofibrantly generated and it has as set of generating cofibrations (resp. trivial cofibrations) $I_C=\{\iota_c(i) |\text{ for every } i\in I, c\in C \}$ 
(resp. $J_C=\{\iota_c(j) |\text{ for every } j\in J, c\in C \}$).

A coloured $\Mm$-enriched operad $\mathcal{O}$ with set of colours $C$ is \emph{admissible} if the corresponding monad ( the one associated to adjunction (\ref{eq.adj.op})) is admissible, i.e. if the model structure on $\Mm^{C}$ can be transferred to the category $\Alg{\mathcal{O}}{\Mm}$ through the (monadic) adjunction (\ref{eq.adj.op}), i.e. if $\Alg{\mathcal{O}}{\Mm}$ can be endowed with a model structure in which a map $f$ is a weak equivalence (fibration) if and only if $U_{\mathcal{O}}(f)$ is a weak equivalence (fibration) in $\Mm^{C}$. Note that if such a model structure is admissible then it is unique and it is cofibrantly generated: the set of generating (trivial) cofibrations is $F_\mathcal{O}(I_C)$ (resp. $F_\mathcal{O}(J_C)$). The transferred model structure on $\Alg{\mathcal{O}}{\Mm}$ is also called the \emph{transferred model structure}.

\begin{defi}\label{d.locK}
Fix a set $C$, a category $\Mm$ and a class of morphisms $K$ in $\Mm$. We will say that a morphism $\morp{f}{A}{B}$ in $\Mm^{C}$ is a \emph{local $K$-morphism} if for every $c\in C$ the morphism $\morp{f(c)}{A(c)}{B(c)}$ belongs to $K$.
\end{defi}

Suppose that $K$ is a saturated class of maps in $\Mm$, let $K'$ be the class of local $K$-morphisms in $\Mm^{C}$; an operad $\mathcal{O}$ as above is \emph{$K$-admissible} in $\Mm$ if the corresponding monad is $K'$-admissible. 

\begin{rmk}
 If $\Mm$ is $K$-compactly generated model category, then $\Mm^{C}$ ( with the induced model structure described above) 
is $K'$-compactly generated.   
\end{rmk}

As a particular case of Proposition \ref{admis.mon} we have: 

\begin{prop}\label{admis.operad}
 Every $K$-admissible operad on a $K$-compactly generated category $\Mm$ is admissible.
\end{prop}

The admissibility of operads is investigated in \cite{BM03}, \cite{BM05}, \cite{Sp01}, \cite{Ha10} , \cite{Mu11} (in the non-symmetric case). 
In \cite{BB13} the question of when a (finitary) polynomial monad is admissible is addressed. 
We recall that one can associate to every coloured operad $\mathcal{O}$  a monad whose category of algebras is the same as the one of $\mathcal{O}$ (that is the monad associated to the adjunction (\ref{eq.adj.op})).
Not all the monads arise from operads but this is the case for finitary polynomial monads (in $\Set$).  In fact there is
an equivalence of categories between the category of finitary polynomial monads in $\Set$ and the category of rigid coloured operads in $\Set$ (\emph{cf.} \cite{Ko11}, \cite{SZ12} for definitions and proofs). A \emph{rigid operad} is just an operad in which the action maps (\ref{eq.action}) act in a free way.
 If $\mathcal{O}_T$ is the $\Set$-operad associated to a certain polynomial monad $T$ the latter is admissible in a monoidal model category $\Mm$ if and only if $o^{\Mm}_\sharp(\mathcal{O}_T)$ is admissible. 

We now give some definitions in order to state some results of interested for us.

In a monoidal category, a saturated class of morphisms which is also closed under tensor products with arbitrary objects is called \emph{monoidally saturated}; the \emph{monoidal saturation} of a class of morphisms $K$ is the minimal monoidally saturated class of morphisms containing $K$.

In $\Mm$ the morphisms belonging to the monoidal saturation of the class of cofibrations $I^{\otimes}$ will be called \emph{$\otimes$-cofibrations}. An object in $\Mm$ is \emph{$\otimes$-small} if it is small respect to the class of $\otimes$-cofibrations. 
The class of weak equivalences is called \emph{$\otimes$-perfect} if it is closed under filtered colimits along $\otimes$-cofibrations. 

A coloured $\Mm$-enriched operad which is $I^{\otimes}$-admissible is said \emph{$\otimes$-admissible}. 
\begin{defi}
A cofibrantly generated monoidal model category $\Mm$ is \emph{compactly generated} if the class of weak equivalences is $\otimes$-perfect and each object is $\otimes$-small (i.e. if it is $I^{\otimes}$-compactly generated).
\end{defi}

We recall that every combinatorial monoidal model category whose class of weak equivalence is closed under filtered colimits is an example of compactly generated monoidal model category. Another example of compactly generated monoidal model category (which is not combinatorial) is given by the category of compactly generated (weak Hausdorff) topological spaces with the standard model structure and the cartesian product (\emph{cf.} \cite[1.2]{BM12}).


We have the following result:

\begin{prop}{(cf. \cite[Theorem 2.1]{BM05})}\label{BMstone}
Suppose $\Mm$ is a cofibrantly generated monoidal model category with cofibrant unit. 
If it admits a monoidal fibrant replacement functor and contains a cocommutative comonoidal interval object, then every symmetric coloured $\Mm$-operad $\mathcal{O}$ has path objects. 

In particular if $\mathcal{O}$ is $K$-admissible and the domains of the generating (trivial) cofibrations are small relative to $K$ for some saturated class of morphisms $K$ in $\Mm$,
then $\mathcal{O}$ is admissible. 
\end{prop}


We refer the reader to \emph{loc. cit.} and \cite{BM03} for the definition of interval object. 

We give the following definition for convenience:
\begin{defi}
A cofibrantly generated model category is \emph{strongly cofibrantly generated} if the domains of the generating cofibrations and trivial cofibrations are small (with respect to the whole category).
\end{defi}
In its original version Theorem 2.1, \cite{BM05} states that in a strongly cofibrantly generated category with a monoidal fibrant replacement functor and a cocommutative comonoidal interval object every operad is admissible. 

Examples of cofibrantly generated monoidal model categories with cofibrant unit admitting a monoidal fibrant replacement functor and a cocommutative comonoidal interval object are:
\begin{itemize}
 \item Simplicial Sets with the Kan-Quillen model structure;
 \item Chain Complexes over a field of characteristic zero with the projective model structure;
 \item Simplicial Modules over any ring, with the model structure transferred form Simplicial Sets;
 \item Compactly generated (weak Hausdorff) topological spaces with the Quillen model structure.
\end{itemize}
 The first three examples are strongly cofibrantly generated (actually, they are all strongly compactly generated), hence every coloured symmetric operad is admissible.
 In the case of (weak Hausdorff) compactly generated topological space the objects are small only respect to $\mathcal{T}_1$, the class of $T_1$ closed inclusions (see Appendix \ref{App.C} for the definition); as a consequence $\mathcal{T}_1$-admissible operads are admissible.

We recall also the monoid axiom, which was first introduced in \cite{SS00} by Schwede-Shipley, in order  to study the admissibility of the associative operad (the operad whose algebras are monoids):
\begin{defi}
A cofibrantly generated monoidal model category $\Mm$ satisfies the \emph{monoid axiom} if any morphism in the monoidal saturated class of trivial cofibrations is a weak equivalence.
\end{defi}

\begin{prop}{\cite[Corollary 8.10]{BB13}}\label{compstone.alg}
Suppose $\Mm$ is a compactly generated monoidal model category and satisfies the monoid axiom, then every tame (finitary) polynomial monad (in $\Set$) is $\otimes$-admissible in $\Mm$ (and hence admissible by Proposition \ref{admis.operad}).\
\end{prop}

We will not give the definition of tame polynomial monad; we will content ourselves with few comments. Let us call an operad in $\Set$ \emph{tame rigid} if its associated monad is tame polynomial. All tame polynomial monads arising from tame rigid operads and all non-symmetric operads are tame rigid. 

We remark that proposition \ref{compstone.alg} was proved for non-symmetric one-coloured operads by Muro in \cite{Mu11}.

Model categories satisfying the hypotheses of Proposition \ref{compstone.alg} are, among others, simplicial sets with Quillen's or Joyal's model structure, compactly generated spaces with Str\o m's or Quillen's model structure, chain complexes over a commutative ring with the projective model structure (\emph{cf.} \cite[Example 1.12, Proposition 2.5]{BB13}).

\section{Bifibrations}
Let us fix a bicomplete symmetric monoidal category $\Mm$.\\ 
The functor $\morp{\Obf}{\ECat{\Mm}}{\Set}$  associates to each $\Mm$-enriched category its set of objects. Similarly the functor $\morp{\Clf}{\EOper{\Mm}}{\Set}$ sends every $\Mm$-enriched coloured operads to its set of colours; we have a similar functor $\Clf$ for non-symmetric coloured operads and reduced coloured operads as well.

The functors $\Obf$ and $\Clf$ are bifibrations (i.e. functors which are both fibrations and opfibrations), giving to $\ECat{\Mm}$, $\EOper{\Mm}$, $\ENSOper{\Mm}$ and $\EROper{\Mm}$ the structure of bifibred categories over $\Set$. We refer the reader to \cite{Ro94} and \cite{Sta12} for the definition of bifibration, in these two papers the presentation of $\ECat{\Mm}$ is also given (note, the second article fixes an error of the first).
Classical references for fibred category are \cite[expos. VI]{SGA1} and \cite[\S 8]{Bo08}.

From the theory of (bi)fibrations we know that there is a $2$-equivalence between the  $2$-category of bifibrations over $\Set$ and the $2$-category of pseudo-functors from $\Set$ to $\Catadj$ (the 2-category which has small categories as objects, adjunctions as 1-morphisms and natural isomorphisms of right adjoint functors as 2-morphisms). Hence we can associate to each bifibration $\morp{\pi}{\mathcal{C}}{\Set}$ a pseudo-functor $\morp{\Fibf{\pi}}{\Set}{\Catadj}$, that determines $\pi$ completely (up to equivalence). 
We recall here how this is done. Fix a bifibration $\morp{\pi}{\mathcal{C}}{\Set}$.
For every $C\in \Set$ let $\Fib{\pi}{C}$ be the \emph{fiber of $\pi$ over $C$}, i.e. the minimal subcategory of $\mathcal{C}$ containing every object $X$ such that $\pi(X)=C$ and every morphism $f$ such that $\pi(f)=\id_C$.

 A \emph{cleavage} for $\pi$ is a choice, for every map of sets $\morp{f}{C}{D}$ and every $A\in \Fib{\pi}{D}$, of an object $f^*(A)\in \Fib{\pi}{C}$ and a cartesian morphism $\morp{\phi_f}{f^*(A)}{A}$ such that $\pi{\phi_f}=f$. 

A cleavege defines a unique functor $\morp{f^*}{\Fib{\pi}{D}}{\Fib{\pi}{C}}$ for every map $\morp{f}{C}{D}$, called the \emph{inverse image functor} of $f$.
  
Dually, a \emph{cocleavage} is a choice, for every map of sets $\morp{f}{C}{D}$ and every $B\in \Fib{\pi}{C}$, of an object $f_!(B)\in \Fib{\pi}{D}$ and a cocartesian morphism $\morp{\upsilon_f}{B}{f_!(B)}$ such that $\pi{\upsilon_f}=f$. A cocleavege defines a unique functor $\morp{f_!}{\Fib{\pi}{C}}{\Fib{\pi}{D}}$ for every map $\morp{f}{C}{D}$, called the \emph{direct image functor} of $f$. 

It follows that for every morphism $\morp{F}{X}{Y}$ in $\mathcal{C}$ there is a unique morphism $\morp{\trf{F}}{C}{Z}$ in $\Fib{\pi}{\pi(X)}$ such that $F=\phi_{\pi{F}}\circ \trf{F}$. Dually there is a unique morphism $\morp{\trc{F}}{W}{Y}$ in $\Fib{\pi}{\pi(Y)}$ such that $F=\trf{F} \circ \upsilon_{\pi{F}}$.

It can easily be proved that the inverse image functor is right adjoint to the direct image functor for every $f$. Equivalently a choice of a cleavage and a left adjoint for every inverse image functor provides a cocleavage. We recall that every bifibration admits a cleavage and a cocleavage.

We can now define the pseudo-functor $\morp{\Fibf{\phi}}{\Set}{\Catadj}$ associated to $\pi$: it sends every set $C$ to the fiber category $\Fib{\pi}{C}$ and every map of set $f$ in the adjunction ($f_!,f^*)$.

We now return to the bifibrations $\Obf$ and $\Clf$. Given $C\in \Set$, the fiber of $\Obf$ over $C$ is $\ECatfc{\Mm}{C}$, the category of $\Mm$-categories with set of objects $C$ and functors between them which are the identity on $C$.

For every map of sets $\morp{f}{C}{D}$ and every $A\in \ECatfc{\Mm}{D}$ we can define the inverse image $f^*(A)\in \ECatfc{\Mm}{C}$ in the following way: for every $x,y\in C$, we set $f^*(A)(x,y)=A(f(x),f(y))$ and the composition is defined using the one of $A$. The unique map $\morp{\phi_f}{f^*(A)}{A}$, such that $\Obf(\phi_f)=f$ and $\phi_f(x,y)=\id_{A(f(x),f(y))}$ is cartesian. The collection of the $\phi_f$'s for every map of sets $f$ and every $A\in \EOperfc{\Mm}{D}$ (where $D$ is the target of $f$) provides a cleavage for $\Obf$.

In a similar way the fiber of $\Clf$ over $C$ coincides with $\EOperfc{\Mm}{C}$ as defined in \ref{col.op}. We can define a cleavage in such a way that the inverse image functor for a map $\morp{f}{C}{D}$ is defined as (\ref{fstardef}).
The same can be done when the domain of $\Clf$ is $\ENSOper{\Mm}$ or $\EROper{\Mm}$, in these cases we denote the fiber over $C$ by $\ENSOperfc{\Mm}{C}$ and $\EROperfc{\Mm}{C}$ respectively.

\newcommand{\EGr}{\mathsf{Gr}}
\newcommand{\EMGr}{\mathsf{MGr}}
\newcommand{\ERMGr}{\mathsf{RMGr}}
There are other (pseudo)functors from $\Set$ to $\Catadj$ that we want to consider:
\begin{itemize}
 \item $\morp{\EGr}{\Set}{\Catadj}$
which sends every set $C$ to the category $\Mm^{C\times C}$ and every map $\morp{f}{C}{D}$ to the adjunction
\[
 \indadjpair{f}{\Mm^{C\times C}}{\Mm^{D\times D}}
\]
induced by the morphism 
$\morp{f\times f}{C\times C}{D\times D}$.
The bifibred category associated to it is $\EGraph{\Mm}$, the category of \emph{$\Mm$-enriched graphs}.

\item $\morp{\EMGr}{\Set}{\Catadj}$ which sends every set $C$ to the category $\Mm^{\Sqofc{C}}$ and every map $\morp{f}{C}{D}$ to the adjunction
\[
 \indadjpair{f}{\Mm^{\Sqofc{C}}}{\Mm^{\Sqofc{D}}}
\]
induced by the morphism $\morp{\Sqofc{f}}{\Sqofc{C}}{\Sqofc{D}}$.
The bifibred category associated to it is $\EMGraph{\Mm}$, the category of \emph{$\Mm$-enriched multi-graphs}.

\item $\morp{\ERMGr}{\Set}{\Catadj}$ which sends every set $C$ to the category $\Mm^{\RSqofc{C}}$ and every map $\morp{f}{C}{D}$ to the adjunction
\[
 \indadjpair{f}{\Mm^{\RSqofc{C}}}{\Mm^{\RSqofc{D}}}
\]
induced by the morphism $\morp{f}{\RSqofc{C}}{\RSqofc{D}}$. The bifibred category associated to it is $\ERMGraph{\Mm}$, the category of \emph{$\Mm$-enriched reduced multi-graphs}.
\end{itemize}
 
We want now to see objects of $\ECat{\Mm}$ (resp. $\EOper{\Mm}$) as  $\Mm$-enriched graphs ( resp. multi-graphs) with structure. We first need the following facts:
\begin{itemize}
\item there is a non-symmetric operad $\Opcat{C}$ such that $\Alg{\Opcat{C}}{\Mm}\cong\ECatfc{\Mm}{C}$. The set of colours of $\Opcat{C}$ is $C\times C$;
\item there is a symmetric operad $\Opop{C}$ such that $\Alg{\Opop{C}}{\Mm}\cong\EOperfc{\Mm}{C}$. The set of colours of $\Opop{C}$ is $\Sqofc{C}$;
\item there is a symmetric operad $\Opnsop{C}$ such that $\Alg{\Opnsop{C}}{\Mm}\cong\ENSOperfc{\Mm}{C}$. The set of colours of $\Opnsop{C}$ is $\Sqofc{C}$;
\item there is a symmetric operad $\Oprop{C}$ such that $\Alg{\Oprop{C}}{\Mm}\cong\EROperfc{\Mm}{C}$. The set of colours of $\Opnsop{C}$ is $\RSqofc{C}$.
\end{itemize}
A description of $\Opcat{C}$ and $\Opop{C}$ can be found in \cite{BM05};
$\Opop{C}$ is also described in detail in \cite{GVR12}, the operads $\Opnsop{C}$ and $\Oprop{C}$ are defined in a similar way.

As special cases of adjunction (\ref{adj.alg}) we get the followings:

\begin{equation}\label{eq.ff.cat}
 \adjpair{F_{\Opcat{C}}}{U_{\Opcat{C}}}{\Mm^{C\times C}}{\ECatfc{\Mm}{C}}
\end{equation}
\begin{equation}\label{eq.ff.op}
 \adjpair{F_{\Opop{C}}}{U_{\Opop{C}}}{\Mm^{\Sqofc{C}}}{\EOperfc{\Mm}{C}}
\end{equation}
\begin{equation}\label{eq.ff.nsop}
 \adjpair{F_{\Opnsop{C}}}{U_{\Opnsop{C}}}{\Mm^{\Sqofc{C}}}{\ENSOperfc{\Mm}{C}}
\end{equation}
\begin{equation}\label{eq.ff.rop}
 \adjpair{F_{\Oprop{C}}}{U_{\Oprop{C}}}{\Mm^{\RSqofc{C}}}{\EROperfc{\Mm}{C}}\ .
\end{equation}

Furthermore for every map of sets $\morp{f}{C}{D}$ there are morphisms of operads $\morp{\Opcat{f}}{\Opcat{C}}{\Opcat{D}}$ and $\morp{\Opop{f}}{\Opcat{C}}{\Opcat{D}}$ (and the same is true if we replace $\Opop{-}$ with $\Opnsop{-}$ or $\Oprop{-}$).
Thanks to (\ref{eq.adj.op}), from each of these morphisms we get an adjunction $(f_!,f^*)$ and a commutative diagram of adjunctions (we picture it only in the case of $\Opop{f}$ but the others are similar)
\begin{equation}\label{diag.fib1}
 \xymatrix{\EOperfc{\Mm}{C} \ar@<1ex>[r]^{f_!} \ar@<1ex>[d]^{U_{\Opop{C}}} & \EOperfc{\Mm}{D} \ar@<1ex>[l]^{f^*} \ar@<-1ex>[d]_{U_{\Opop{D}}}\\ 
                                \Mm^{\Sqofc{C}} \ar@<-1ex>[r]_{f_!} \ar@<1ex>[u]^{F_{\Opop{C}}}& \Mm^{\Sqofc{D}} \ar@<-1ex>[l]_{f^{*}} \ar@<-1ex>[u]_{F_{\Opop{D}}}}
\end{equation}
and the horizontal maps turn out to be exactly the direct-inverse image adjunctions associated to the corresponding bifibration. This amounts to say that we have fibred adjunctions
 \begin{equation}\label{eq.ff.cat2}
 \adjpair{F_{\Opcat{}}}{U_{\Opcat{}}}{\EGraph{\Mm}}{\ECat{\Mm}}
\end{equation}
\begin{equation}\label{eq.ff.op2}
 \adjpair{F_{\Opop{}}}{U_{\Opop{}}}{\EMGraph{\Mm}}{\EOper{\Mm}}
\end{equation}
\begin{equation}\label{eq.ff.nsop2}
 \adjpair{F_{\Opnsop{}}}{U_{\Opnsop{}}}{\EMGraph{\Mm}}{\ENSOper{\Mm}}
\end{equation}
\begin{equation}\label{eq.ff.rop2}
 \adjpair{F_{\Oprop{}}}{U_{\Oprop{}}}{\ERMGraph{\Mm}}{\EROper{\Mm}}
\end{equation}
that fiber-wise are given by (\ref{eq.ff.cat}), (\ref{eq.ff.op}), (\ref{eq.ff.nsop}) and (\ref{eq.ff.rop}) respectively.

For every $C\in \Set$ there is also a morphism of operads $\morp{j_C}{\Opcat{C}}{\Opop{C}}$ which induces an adjunction between the categories of algebras
\begin{equation}\label{eq.adj.ff.cat.op}
 \indadjpair{j_C}{\ECatfc{\Mm}{C}}{\EOperfc{\Mm}{C}}
\end{equation}
The functor $j_{C!}$ is the usual inclusion of the category of categories with $C$ as set of objects in the category of $C$-coloured operads.

Furthermore for every map of sets $\morp{f}{C}{D}$ the following diagram of adjunctions is commutative 
\begin{equation}\label{diag.fib2}
 \xymatrix{\EOperfc{\Mm}{C} \ar@<1ex>[r]^{f_!} \ar@<1ex>[d]^{\indr{j_C}} & \EOperfc{\Mm}{D} \ar@<1ex>[l]^{f^*} \ar@<-1ex>[d]_{\indr{j_D}}\\ 
                                \ECatfc{\Mm}{C} \ar@<-1ex>[r]_{f^{\sharp}_!} \ar@<1ex>[u]^{\indl{j_C}} & \ECatfc{\Mm}{D} \ar@<-1ex>[l]_{f^{\sharp *}} \ar@<-1ex>[u]_{\indl{j_D}}}
\end{equation}

so we get an adjunction of fibred categories
\begin{equation}
 \indadjpair{j}{\ECat{\Mm}}{\EOper{\Mm}}
\end{equation}
which fiber-wise is given by \ref{eq.adj.ff.cat.op}.
The functor $j_!$ is just the usual fully faithful inclusion of $\ECat{\Mm}$ in $\EOper{\Mm}$.
Similarly we have morphisms of operads from $\Opcat{C}$ to $\Opnsop{C}$ and $\Oprop{C}$, which give us adjunctions:
\begin{equation}
 \indadjpair{j}{\ECat{\Mm}}{\ENSOper{\Mm}}
\end{equation}
\begin{equation}
 \indadjpair{j}{\ECat{\Mm}}{\EROper{\Mm}}
\end{equation}

\subsection{Homotopy theory of the fibers}

Assume that $\Mm$ is a cofibrantly generated monoidal model category (\emph{cf.} \cite{Ho99}) and let $I,J$ be the sets of generating cofibrations and generating trivial cofibrations
respectively.

In order to prove the existence of the canonical model structure on $\ECat{\Mm}$, $\EOper{\Mm}$, $\ENSOper{\Mm}$ or $\EROper{\Mm}$ we want that all the fibers of the corresponding bifibration carry the transferred model structure. For this reason we give the following definitions:

\begin{defi}\label{opadefi} A cofibrantly generated monoidal model category $\Mm$ with a distinguished saturated class of morphisms $K$:

\begin{enumerate}
\item \emph{admits transfer ($K$-transfer) for categories} if for every $C\in \Set$ the operad $\Opcat{C}$ is admissible (resp. $K$-admissible) in $\Mm$; 
\item \emph{admits transfer ($K$-transfer) for (symmetric) operads} if for every $C\in \Set$ the operad $\Opop{C}$ is admissible (resp. $K$-admissible) in $\Mm$;
\item \emph{admits transfer ($K$-transfer) for non-symmetric operads} if for every $C\in \Set$ the operad $\Opnsop{C}$ is admissible (resp. $K$-admissible) in $\Mm$ (hence admissible);
\item \emph{admits transfer ($K$-transfer) for reduced operads} if for every $C\in \Set$ the operad $\Oprop{C}$ is admissible (resp. $K$-admissible) in $\Mm$.
\end{enumerate}
\end{defi}

\begin{rmk}
As a consequence of Proposition \ref{admis.operad}, if a $K$-compactly generated monoidal category $\Mm$ admits $K$-transfer (and transfer) for operads (non-symmetric operads, reduced operads, categories) then it admits transfer for operads (resp. non-symmetric operads, reduced operads, categories).
\end{rmk}

Note that if $\Mm$ admits transfer ($K$-transfer) for (non-symmetric, reduced) operads then it admits transfer ($K$-transfer) for categories.

Observe that if $\Mm$ admits transfer for (non-symmetric, reduced) operads then all the adjunctions in diagrams (\ref{diag.fib1}) and (\ref{diag.fib2}) are Quillen adjunctions.

The operad $\Opcat{C}$ is non-symmetric. The operads $\Opop{C}$, $\Opnsop{C}$ and $\Oprop{C}$ are only rigid, but only the last two are tame (\emph{cf.} \cite[section 11-12]{BB13}).

As special cases of Propositions \ref{BMstone} and \ref{compstone.alg} we get:
\begin{prop}\label{compstone}
Suppose $\Mm$ is a cofibrantly generated monoidal model category:
\begin{enumerate}
\item if $\Mm$ is a strongly cofibrantly generated with cofibrant unit, admits a monoidal fibrant replacement functor and contains a cocommutative comonoidal interval object, then it admits transfer for symmetric operads. 

\item if $\Mm$ is compactly generated monoidal model category satisfying the monoid axiom, then $\Mm$ admits $\otimes$-transfer for non-symmetric operads and reduced operads.
\end{enumerate}
\end{prop} 


To conclude we would like to make few remarks about the generating (trivial) cofibrations of the fibers.
Suppose that $\Mm$ admits transfer for symmetric operads. We know from section \ref{sec.alg} that the sets of generating cofibrations and trivial cofibrations of $\EOperfc{\Mm}{C}$ are:
\[
 I_{\Opop{C}}=\{F_C (\iota_s(i)) |\text{ for every } i\in I, s\in \Sqofc{C} \}
\] 
and
\[
 J_{\Opop{C}}=\{F_C (\iota_s(j)) |\text{ for every } j\in J, s\in \Sqofc{C} \}
\]

For every $n\in \N$ let $s_n\in \Sqofc{[n+1]}$ be the signature
$(1,2,\dots,n;0)$.
Observe that for every set $C$ and every signature $s=(c_1,\dots,c_n; c_0)$ in $C$ of valence $n$ there is a map $\morp{f_s}{[n]}{C}$ (sending $i$ to $c_i$) such that $f_s (s_n)=s$.

 As a consequence $\iota_s=f_{s!}\iota_{s_n}$. Conversely for every map $\morp{f}{[n]}{C}$ there exists an $s\in \Sqofc{C}$   such that $\iota_s=f_! \iota_{s_n}$.
This implies that the set of generating (trivial) cofibrations of $\Mm^{\Sqofc{C}}$ can be defined as $\{f^{\sharp}_!(\iota_{s_n}(i)) |\text{ for every } i\in I, n\in \N, \morp{f}{[n]}{C} \}$ (resp. $\{f^{\sharp}_!(\iota_{s_n}(j)) |\text{ for every } j\in J, n\in \N, \morp{f}{[n]}{C} \}$).
So we can rewrite:
\[
 I_{\Opop{C}}=\{f_!F_{[n]}(\iota_{s_n}(i)) |\text{ for every } i\in I, n\in \N, \morp{f}{[n]}{C} \}
\] 
and
\[
 J_{\Opop{C}}=\{f_!F_{[n]}(\iota_{s_n}(i)) |\text{ for every } j\in J, n\in \N, \morp{f}{[n]}{C} \}
\]
We will shorten the functor $F_{[n]}\iota_{s_n}$ in $C_n$.

\begin{rmk}\label{rmk1}
Note that to give a commutative diagram of type:
\[
 \diagc{C_n(X)}{A}{C_n(Y)}{B}{C_n(i)}{a}{f}{b}{}
\]
in $\EOper{\Mm}$ is the same as to give a diagram:
\[
 \diagc{a_! C_n(X)}{A}{a_!C_n(Y)}{f^*(B)}{a_! C_n(i)}{\trc{a}}{\trf{f}}{\phi_a \trf{b}}{}
\]
in $\EOperfc{\Mm}{\Cl{A}}$.

This implies that the map $\trf{f}$ is a fibration (trivial fibration) in $\EOperfc{\Mm}{\Cl{A}}$ if and only if $f$ has the right lifting property respect to
$C_n(i)$ for every $n\in \N$ and $i\in I$ (resp. $i\in J$). 

Similar considerations apply to for $\ENSOper{\Mm}$ and $\EROper{\Mm}$. 
\end{rmk}

\section{The model structure}
\subsection{Weak equivalences and fibrations}
Let us fix a monoidal model category $(\Mm,\otimes,I)$. 
In this section we want to define the classes of weak equivalences and fibrations for the model structure that we want to establish on $\EOper{\Mm}$.

Let $\mathbb{I}$ be the $\Mm$-category with set of objects $\{0,1\}$ representing a single isomorphism, i.e. $\mathbb{I}(0,0)=\mathbb{I}(0,1)=\mathbb{I}(1,0)=\mathbb{I}(1,1)=I$.

We recall the definition of $\Mm$-interval given in \cite{BM12}:
\begin{defi}
Given a monoidal model category $\Mm$ which admits transfer for categories, a \emph{$\Mm$-interval} is a cofibrant object in $\ECatfc{\Mm}{\{0,1\}}$ (with the transferred model structure), weakly equivalent to $\mathbb{I}$.

A set $\mathfrak{G}$ of $\Mm$-intervals is \emph{generating} if each $\Mm$-interval $\mathbb{H}$ is a retract of a trivial extension $\mathbb{K}$ of a $\Mm$-interval $\mathbb{G}$ in $\mathfrak{G}$, i.e. there is a diagram in $\ECatfc{\Mm}{\{0,1\}}$

\[
\xymatrix{\mathbb{G}\ar@{>->}[r]^{j} & \mathbb{K} \ar@<0.5ex>[r]^{r} &\mathbb{H} \ar@<0.5ex>[l]^{i}}
\]

such that $j$ is a trivial cofibration and $ri=\id_{\mathbb{H}}$. 
\end{defi}
\begin{defi}\label{amalgadef}
Let $\morp{\partial_i}{\{0,1\}}{\{0,1,2\}}$ denote the order-preserving inclusion which omits $i$. The \emph{amalgamation} of two objects
$\mathbb{H}$ and $\mathbb{K}$ in $\ECatfc{\Mm}{\{0,1\}}$ is defined as $\mathbb{H}\ast\mathbb{K}=\partial_1^*(\partial_{2!}\mathbb{K}\cup \partial_{0!}\mathbb{H})$,
where the coproduct is taken in $\ECatfc{\Mm}{\{0,1,2\}}$. 
\end{defi}

Note that $\partial_{2!}\mathbb{K}\cup \partial_{0!}\mathbb{H}$ is isomorphic to the colimit of the following diagram in $\ECat{\Mm}$:

\[
\xymatrix{
 \obCat \ar[r]^{i_0}\ar[d]^{i_1} & \mathbb{H}\\
 \mathbb{K}}
\]
where $\obCat$ is the $\Mm$-category representing a single object, i.e. the one with one object $\ast$ and $\obCat(\ast,\ast)=I$; The morphism $i_j$ is the one determined by the object $j$ of the target category.  

It is clear that, given a $\Mm$-category $C$ and two morphisms $\morp{h}{\mathbb{H}}{C}$ and $\morp{k}{\mathbb{K}}{C}$ such that $h(0)=k(1)$, there is a morphism $\morp{a}{\mathbb{H}\ast\mathbb{K}}{C}$ such that $a(0)=k(0)$ and $a(1)=h(1)$.

We also recall the following definitions from \cite[2.2]{BM12}:
\begin{defi}\label{def.mod}
A $\Mm$-functor (i.e. a morphism in $\ECat{\Mm}$)  $\morp{F}{A}{B}$ is said to be:
\begin{itemize}
 \item \emph{path-lifting} if it has the right lifting property with respect to $\morp{i_k}{\obCat}{\mathbb{H}}$ for any $k\in \{0,1\}$ and any $\Enr{\Mm}$-interval
$\mathbb{H}$ (see definition \ref{amalgadef} and the discussion below for the definition of $\obCat$ and the $i_k$'s).
\item \emph{essentially surjective} if for any object $\morp{b}{\obCat}{B}$ there is an object $\morp{a}{\obCat}{A}$, a $\Enr{\Mm}$ interval $\mathbb{H}$
and a commutative diagram
\[
 \xymatrix{\obCat \ar[rr]^-{a}\ar[rd]^-{i_0}& & A \ar[dd]^{F}\\
           & \mathbb{H}\ar[rd] & \\
           \obCat \ar[rr]^-{b}\ar[ru]^-{i_1}& & B}
\] 
\end{itemize}

\end{defi}

\begin{defi}
 Two object $a_0,a_1$ of a $\Enr{\Mm}$ category $A$ are 
\begin{itemize}
\item \emph{equivalent} if there exists a $\Enr{\Mm}$ interval $\mathbb{H}$  and a 
$\morp{\gamma}{\mathbb{H}}{A}$ such that $\gamma(0)=a_0$ and $\gamma(1)=a_1$;
\item \emph{virtually equivalent} if they become equivalent in a some fibrant replacement $A_f$ of $A$ in $\ECatfc{\Mm}{\mathbf{Ob}(A)}$;
\item \emph{homotopy equivalent} if there exist maps
\[
 \morp{\alpha}{I}{A_f(a_0,a_1)}\text{ and } \morp{\beta}{I}{A_f(a_1,a_0)}
 \]
such that $\morp{\beta\alpha}{I}{A_f(a_0,a_0)}$
(resp. $\morp{\beta\alpha}{I}{A_f(a_1,a_1)}$) is homotopic to the arrow $\arr{I}{A_f(a_0,a_0)}$ (resp. $\arr{I}{A_f(a_1,a_1)}$) defining the identity of $a_0$ (resp. $a_1$).
\end{itemize}   
\end{defi}

Suppose now that $\Mm$ admits transfer for symmetric operads (non-symmetric operads, reduced operads) and $K$ is a class of morphisms in $\Mm$. We can give the following definitions:
\begin{defi}\label{defmod}
a morphism $\morp{f}{P}{Q}$  in $\EOper{\Mm}$ (resp. $\ENSOper{\Mm}$, $\EROper{\Mm}$) is said to be:
\begin{itemize}
\item a \emph{local fibration (weak equivalence, trivial fibration, local $K$-morphism)} if the corresponding morphism $\morp{\trf{f}}{P}{f^*(Q)}$ is a fibration (weak equivalence, trivial fibration, local $K$-morphism) in $\EOperfc{\Mm}{\Cl{P}}$.
 \item \emph{path-lifting} if it has the right lifting property respect to $j_!(\{ i\} \longrightarrow \mathbb{H})$, $i=0,1$ for any $\Mm$-interval $\mathbb{H}$
(equivalently $j^*(f)$ is path-lifting).
\item \emph{essentially surjective} if $j^*(f)$ is essentially surjective;
\item a \emph{fibration} if it is a path-lifting local fibration;
\item a \emph{weak equivalence} if it is an essentially surjective local weak equivalence; 
\item \emph{fully faithful} if the corresponding morphism $\morp{\trf{f}}{P}{f^*(Q)}$ is an isomorphism in $\EOperfc{\Mm}{\Cl{P}}$.
\end{itemize}
\end{defi}
\begin{rmk}\label{rmk2}
Note that from Remark \ref{rmk1} and the characterization of fibrations in the transferred model structure on $\EOperfc{\Mm}{\Cl{P}}$ it follows that a morphism $\morp{f}{P}{Q}$ is a local fibration (trivial fibration) if and only if
it has the right lifting property respect to $J_{\mathrm{loc}}=\{C_n(j)|\  j\in J \ n\in \N \}$ (risp. $I_{\mathrm{loc}}=\{C_n(i)|\  i\in I \ n\in \N \}$).
 
\end{rmk}

We will give the proof of the following lemmas only in the case of symmetric operads, since the proofs in the non-symmetric and reduced cases are identical.

\begin{lemma}
A locally fibrant $\EOper{\Mm}$ (resp. $\ENSOper{\Mm}$, $\EROper{\Mm}$) is fibrant. 
\end{lemma}
\begin{proof}
If $P$ is locally fibrant in $\EOper{\Mm}$ then $j^*(P)$ is locally fibrant in $\ECat{\Mm}$ and hence it is path-lifting ( \cite[Lemma 2.3]{BM12}).   
\end{proof}

\begin{lemma}\label{lemma2.4}
A morphism of $\EOper{\Mm}$ (resp. $\ENSOper{\Mm}$, $\EROper{\Mm}$) is a trivial fibration if and only if it is a local trivial fibration which is surjective on colours. 
\end{lemma}
\begin{proof}
A morphism $f$ in $\EOper{\Mm}$ is a trivial fibration if and only if it is a local trivial fibration and $j^{*}(f)$ is path-lifting and essentially surjective.
Since $j^{*}(f)$ is a local fibration if $f$ is so and $f$ is surjective on colours if and only if $j^{*}(f)$ is surjective on objects the statement follows from \cite[Lemma 2.4]{BM12}.   
\end{proof}
\subsection{(Virtual) equivalence of colours}
We continue to assume that $\Mm$ admits transfer for symmetric operads (non-symmetric operads or reduced operads, depending on the case we are interested in). The goal of this section is to prove Proposition \ref{prop2.12}, i.e. the 2-out-of-3 property for weak equivalences. We are going to give definitions and proofs only in the case of symmetric operads, since the arguments in the non-symmetric and reduced cases are almost identical.

\begin{defi}
 
Two colours $c_0,c_1$  of a $\Mm$-Operad $P$ are: 
\begin{itemize}
 \item \emph{equivalent} if they are equivalent as objects of $j^*(P)$ (in the sense of \cite{BM12});
 \item \emph{virtually equivalent} if they become virtually equivalent in some fibrant replacement $P_{f}$ of $P$ in $\EOperfc{\Mm}{\Cl{P}}$;
 \item \emph{homotopy equivalent} if they are homotopy equivalent as objects of $j^*(P)$. The relation of homotopy equivalence will be denoted by $\sim$.
\end{itemize}

\end{defi}

From Lemma 2.8 in \cite{BM12} applied to $j^*(P)$ we get immediately the following:

\begin{lemma}\label{lemma2.7}
For any $\Enr{\Mm}$ operad $P$, equivalence and virtual equivalence are equivalence relations on $\Cl{P}$.
\end{lemma} 

\begin{lemma}\label{lemma2.8}
A local weak equivalence of operads $\morp{F}{P}{Q}$ reflects virtual equivalence of colours, i.e. for every $c,d\in \Cl{P} $ if $F(c)$ and $F(d)$ are virtually equivalent in $Q$, then
$c$ and $d$ are so in $P$.
\end{lemma}
\begin{proof}
 Since $j^*$ preserves local weak equivalences, it follows directly from \cite[Lemma 2.9]{BM12}.
\end{proof}

\begin{lemma}\label{lemma2.9}
 If $\Mm$ is right proper then for any $\Enr{\Mm}$ operad $P$, virtually equivalent colours of $P$ are equivalent.
\end{lemma}

\begin{proof}
 It follows from \cite[Lemma 2.10]{BM12} since if $P_f$ is a fibrant replacement for $P$ then $j^*(P_f)$ is a cofibrant replacement for $j^*(P)$ ($j_{\Cl{P}}^*$ preserves weak equivalences and fibrant objects hence it preserves fibrant replacements).
\end{proof}

\begin{lemma}\label{lemma2.10}
 In any $\Enr{\Mm}$ operad $P$ virtually equivalent colours are homotopy equivalent.
\end{lemma}

\begin{proof}
 if $c_0, c_1$ are virtually equivalent colours in $P$ then they are virtually equivalent as objects of $j^*(P)$ thus they are homotopy equivalent by Lemma 2.11 in \cite{BM12} 
\end{proof}


\begin{lemma}\label{lemma2.11}
Given $n\in \N$, $P\in \EOper{\Mm}$ and $c_0,\dots c_n, d_0, \dots, d_n \in \Cl{P}$ such that $c_0 \sim d_0, c_1 \sim d_1 ,\dots , c_n \sim d_n$  and $c\sim d$.
Then $P(c_0,\dots, c_n; c)$ and $P(d_0,\dots, d_n; d)$ are related by a zig-zag of weak equivalences in $\Mm$. Moreover any morphism of $\Enr{\Mm}$ operads $\morp{f}{P}{Q}$
induces a functorially related zig-zag of weak equivalence between $Q(f(c_0),\dots, f(c_n); f(c))$ and $Q(f(d_0),\dots, f(d_n); f(d))$.   
\end{lemma}

\begin{proof}
 We can prove the statement in the case in which $c_1=d_1,\dots, c_n=d_n$, the general case will follow by iteration.


By assumption we have a fibrant replacement $P_f$ of $P$ in $\EOperfc{\Mm}{\Cl{P}}$ and arrows $\morp{\alpha}{\Un_{\Mm}}{P_f(c_0,d_0)}$, $\morp{\beta}{\Un_{\Mm}}{P_f(d_0,c_0)}$ 
such that $\morp{\beta \circ \alpha}{\Un_{\Mm}}{P_f(c_0,c_0)}$ (resp. $\morp{\beta \circ \alpha}{\Un_{\Mm}}{P_f(d_0,d_0)}$) is homotopic to the arrow $\morp{\id_{c_0}}{\Un_{\Mm}}{P_f(c_0,c_0)}$ (resp. $\morp{\id_{d_0}}{\Un_{\Mm}}{P_f(d_0,d_0)}$) given by the identity.
Similarly we have arrows $\morp{\alpha'}{\Un_{\Mm}}{P_f(c,d)}$, $\morp{\beta'}{\Un_{\Mm}}{P_f(d,c)}$ such that $\beta'\circ \alpha'$ is homotopic to $\id_c$ and $\alpha'\circ \beta'$ is homotopic to $\id_d$.

Using the internal composition of $P_f$ we obtain two morphisms
\[
 \morp{(\alpha,\id_{c_1},\dots, id_{c_n})^*}{P_f(d_0,c_1,\dots, c_n; c)}{P_f(c_0,c_1,\dots, c_n; c)}
\]
\[
 \morp{(\beta,\id_{c_1},\dots, id_{c_n})^*}{P_f(c_0,c_1,\dots, c_n; c)}{P_f(d_0,c_1,\dots, c_n; c)}
\]
which become mutually inverse isomophisms in the homotopy category $\HoC{\Mm}$, hence they are weak equivalences.
Similarly the morphisms
\[
 \morp{\alpha'_*}{P_f(d_0,c_1,\dots, c_n; c)}{P_f(d_0,c_1,\dots, c_n; d)}
\]
\[
 \morp{\beta'_*}{P_f(d_0,c_1,\dots, c_n; d)}{P_f(d_0,c_1,\dots, c_n; c)}
\]
are also weak equivalences since they are mutually inverse isomorphisms in the homotopy category.\\
The zig-zag of weak equivalences between $P(c_0,c_1,\dots, c_n; c)$ and $P(d_0,c_1,\dots, c_n; d)$  is given by
 \[
  \morp{\alpha'_* (\beta,\id_{c_1},\dots, id_{c_n})^*}{P_f(c_0,c_1,\dots, c_n; c)}{P_f(d_0,c_1,\dots, c_n; d)}
 \]
concatenated with the weak equivalences $\arro{\sim}{P(c_0,c_1,\dots, c_n; c)}{P_f(c_0,c_1,\dots, c_n; c)}$ and $\arro{\sim}{P(d_0,c_1,\dots, c_n; d)}{P_f(d_0,c_1,\dots, c_n; d)}$.

Every morphism of $\Enr{\Mm}$operads $\morp{g}{P}{Q}$ induces a morphism between some fibrant replacements $P_f$ and $Q_f$
\[
 \diagc{P}{P_f}{g^*(Q)}{g^*(Q_f)}{\trf{g}}{\sim}{\trf{g}_f}{\sim}{} \diagc{P}{P_f}{Q}{Q_f}{f}{\sim}{g_f}{\sim}{}\ .
\]

The following diagram commutes (here $g(\alpha')$ is the composition of $\alpha'$ with $\morp{g_f}{P_f(c,d)}{Q_f(g(c),g(d))}$, etc.) 
\[
 \xymatrix{P_f(c_0,c_1,\dots, c_n; c) \ar[d]_{g_f} \ar[r]^{\alpha'_* (\beta,\id_{c_1}\dots, id_{c_n})^*} & P_f(d_0,c_1,\dots, c_n; d) \ar[d]^{g_f}\\
Q_f(g(c_0),g(c_1),\dots, g(c_n); g(c)) \ar[r]_*!/d5pt/{g(\alpha')_* (g(\beta),\id_{c_1},\dots, id_{c_n})^*} & Q_f(g(d_0),g(c_1),\dots, g(c_n); g(d))
  }
\]

and it is easy to check that there is a functorially related zig-zag of weak equivalences between $Q(g(c_0),g(c_1),\dots, g(c_n); g(c))$ and $Q(g(d_0),g(c_1),\dots, g(c_n); g(d))$.
\end{proof}

\begin{prop}\label{prop2.12}
If $\Mm$ is right proper, the class of weak equivalences in $\EOper{\Mm}$ ($\ENSOper{\Mm}$, $\EROper{\Mm}$) satisfies the 2-out-of-3 property. 
\end{prop}

\begin{proof}
 Let $\morp{f}{P}{Q}$ and $\morp{g}{Q}{R}$ be morphisms of $\Enr{\Mm}$ operads.

We have to prove that if two among $f$, $g$ and $fg$ are weak equivalences (i.e. essentially surjective local weak equivalences) then so is the remaining one. We are going to prove the three cases separately:

\begin{itemize}
 \item Assume $f$ and $g$ are weak equivalences. It is easy to check that $gf$ is a local weak equivalence.
 The fact that $gf$ is essentially surjective follows from the fact that $j^*(gf)$ is (\emph{cf.} \cite[Proposition 2.13]{BM12} ); 
\item Assume $f$ and $gf$ are weak equivalences. It is immediate to check that $g$ is essentially surjective. What is left to check is that $g$ is a local weak equivalence.
Given $d_0,d_1,\dots, d_n, d \in \Cl{Q}$ since $f$ is essentially surjective we can pick $c_0,c_1,\dots, c_n, c \in \Cl{Q}$ such that $f(c_i)$ is equivalent to $d_i$ for every $i\in \{0,\dots,n\}$ and $f(c)$ is equivalent to
$d$. It follows from Lemma \ref{lemma2.10} that $f(c_i)\sim d_i$ for every $i\in \{0,\dots,n\}$ and $f(c)\sim d$ hence Lemma \ref{lemma2.11} gives us a zig-zag of weak equivalences (in $\Mm$) between
$Q(f(c_0),\dots,f(c_n); f(c))$ and $Q(d_0,\dots,d_n; d)$. We get the following commutative diagrams:
\[
\xymatrix@C=3cm{P(Q(c_0,\dots, c_1; c) \ar[dr]_*!/l10pt/{gf(c_0,\dots, c_1; c)} \ar[r]^{\!\!f(c_0,\dots, c_1; c)} &  
          Q(f(c_0),\dots, f(c_n); f(c)) \ar[d]^{g(f(c_0),\dots, f(c_n); f(c))} \\
          &
          R(gf(d_0),\dots, gf(d_n); gf(d)) 
         }
\]

\[ 
\xymatrix{Q(f(c_0),\dots, f(c_n); f(c)) \ar[d]^{g(f(c_0),\dots, f(c_n); f(c))} \ar @{-} [r]^{\sim} &
          Q(d_0,\dots, d_n; d) \ar[d]^{g(d_0,\dots, d_n; d)} \\
          R(gf(d_0),\dots, gf(d_n); gf(d)) \ar @{-} [r]^{\sim} &
          R(g(d_0),\dots, g(d_n); g(d))
          }
\]
(the zig-zag of weak equivalences on the bottom row of the last diagram is obtained from the one on the top row applying Lemma \ref{lemma2.11}).\\
The vertical arrow in the first diagram is a weak equivalence since the other two maps in the diagram are so (by assumption). Applying the 2-out-of-3 property of weak equivalences
in $\Mm$ in the second diagram we get that $g(d_0,\dots, d_n; d)$ is a weak equivalence;

\item Assume $g$ and $gf$ are weak equivalences. It is immediate to verify that $f$ is a local weak equivalence too. It follows from Lemmas \ref{lemma2.8} and
\ref{lemma2.9} that $g$ reflects equivalences of objects thus the essentially surjectivity of $f$ follows from the one of $gf$.
\end{itemize}
\end{proof}

\subsection{Generating (trivial) cofibrations}

Lemma \ref{lemma2.4} implies that the class of the trivial fibrations (as defined in \ref{def.mod}) in $\EOper{\Mm}$ ($\ENSOper{\Mm}$, $\EROper{\Mm}$) is characterized by the right lifting property respect to:
\begin{equation}\label{eq.gen.c}
 \mathcal{I}=I_{loc}\cup \{\arr{\emptyset}{j^*(\obCat)}\}
\end{equation}
where $\obCat$ is the $\Mm$-category representing a single object (as in Definition \ref{amalgadef}) and $I_{loc}$ is defined as in Remark \ref{rmk2}. The morphisms which are surjective on the colours
are exactly the ones having the right lifting property respect to the last map. The set $\mathcal{I}$ is therefore a good candidate for the set of generating cofibrations for the model structure we want to establish.

If there exists a set $\mathfrak{G}$ of generating $\Mm$-intervals we also have a good candidate for the set of generating trivial cofibrations, that is
\begin{equation}\label{eq.gen.tc}
 \mathcal{J}=J_{loc}\cup \{ \morp{j_!(i_1)}{j_!(\obCat)}{j_!(\mathbb{G})}\mid \mathbb{G}\in \mathfrak{G}\}\ .
\end{equation}
In fact, by definition, a morphism $f$ in $\EOper{\Mm}$ is path-lifting if and only if it has the right lifting property respect to $\morp{j_!(i_1)}{j_!(\obCat)}{j_!(\mathbb{H})}$
for every $\Mm$-interval $\mathbb{H}$, but if $f$ is locally fibrant it is sufficient to check it only for the $j_!(i_1)$'s coming from generating $\Mm$-intervals. In fact, suppose that $\mathbb{H}$ is a retract of a trivial extension of a generating $\Mm$-interval $\mathcal{G}$ and that the following diagram commutes:
\[
 \xymatrix
{ & & j_!(\obCat) \ar[dll]_{j_!(i_1)} \ar[d]^{j_!(i_1)} \ar[r]^{a} & A\ar[d]^{f}\\
  j_!(\mathbb{G}) \ar@{-->}[urrr]^{l} \ar@{>->}[r]^{e} & j_!(\mathbb{K}) \ar@<0.5ex>[r]^{r} & j_!(\mathbb{H}) \ar@<0.5ex>[l]^{i} \ar[r]^{h} & B 
}\ .
\]
If $l$ exists, it can be extended to a map $\morp{l'}{K}{A}$ since $e$ is a trivial cofibration (in $\EOperfc{\Mm}{\{0,1\}}$ and $f$ is locally fibrant. The map $l'i$ gives the desired lifting for $\mathbb{H}$.
 
We want now to be sure that the domains of $\mathcal{I}$ (resp. $\mathcal{J}$) are small with respect to $\mathcal{I}$ (resp. $\mathcal{J}$), so that we can apply the small object argument. We now have to distinguish between the case of $\EOper{\Mm}$ and the ones of $\ENSOper{\Mm}$ and $\EROper{\Mm}$. 

\begin{lemma} \label{ess.surj.lemma}
Let $\Mm$ be a model category admitting transfer for categories.
In $\ECat{\Mm}$ the transfinite composition of essentially surjective morphisms is essentially surjective.
\end{lemma}
\begin{proof}

Let $\alpha$ be an ordinal (regarded as a category) and let $\morp{F}{\alpha}{\ECat{\Mm}}$ be a functor defining a transfinite composition of essentially surjective morphisms.
Let $A=F(0)$, $B=\varinjlim_{} F$ and $\morp{f}{A}{B}$ be the transfinite composition obtained by $F$.
Given an object $\morp{b}{\obCat}{B}$ we want to find an object $\morp{a}{\obCat}{A}$, an interval $\mathbb{H}$ and morphism $\morp{h}{\mathbb{H}}
{B}$ such that $h(0)=f(a)$ and $h(1)=b$. 
 
The proof will proceed by transfinite induction on $\alpha$. 

If $\alpha$ is the ordinal with one object (as a category) then there is nothing to prove (the transfinite composition is just the identity).

Suppose $\alpha$ is a limit ordinal and the statement is true for any $\beta < \alpha$.

In $\ECat{\Mm}$ the object $\obCat$ is small so $b$ factors through some $F(\beta)$ where $\beta$ is a successor ordinal smaller then 
$\alpha$, i.e. there is $b'\in F(\beta)$ such that $f_{\beta\alpha}(b')=b$, where $f_{\beta\alpha}$ is the canonical morphism from $F(\beta)$ to $B$. 

$F$ restricted to $\{ \gamma \leq \beta\}$ defines a $\beta$-transfinite composition $f_{\beta}$ of essentially surjective morphisms, which is essentially surjective by inductive hypothesis.

We can then find an interval $\mathbb{H}$, an object $a$ in $A$ and a morphism $\morp{h'}{\mathbb{H}}{F(\beta)}$ such that $h'(0)=f_{\beta}(a)$ and $h'(1)=b'$. 
The desired morphism is then $h=f_{\beta \alpha} \ h'$. 

For the successor case suppose $\alpha=\beta +1$ and that the statement is true for every transfinite composition indexed by $\beta$.
Factor $f$ through $F(\beta)$:
\[
 \xymatrix{A\ar@/_1pc/[rr]_{f} \ar[r]^{f_{\beta}} & F(\beta) \ar[r]^{f_{\beta \alpha}} & F(\alpha)=B}
\]
By hypothesis $f_{\alpha \beta}$ is essentially surjective so we can find an object $b'\in F(\beta)$, an interval $\mathbb{H}'$ and a morphism 
$\morp{h'}{\mathbb{H}'}{B}$ such that $h'(0)=f_{\alpha \beta}(b')$ and $h'(1)=b$.

On the other hand $f_{\beta}$ is also essentially surjective by inductive hypothesis so we can find $a\in A$, an interval $\mathbb{H}''$ and a morphism 
$\morp{h''}{\mathbb{H}''}{F(\beta)}$ such that $h''(0)=f_{\beta}(a)$ and $h''(1)=b'$.

Consider the amalgamation $\mathbb{H}''* \mathbb{H}'$, since $f_{\beta \alpha} h''(1)= f_{\beta \alpha} h''(1) =h'(0)$ we get a morphism
$\morp{h}{\mathbb{H}''* \mathbb{H}'}{B}$ such that $h(0)=f_{\beta \alpha} h''(0)=f(a)$ and $h(1)=h'(1)=b$ by the universal property of push-outs. Consider a cofibrant replacement $\morp{\iota}{\overline{\mathbb{H}''* \mathbb{H}'}}{\mathbb{H}''* \mathbb{H}'}$ in $\ECatfc{\Mm}{\{0,1\}}$. By Lemma 1.16 \cite{BM12} the $\Enr{\Mm}$category  $\overline{\mathbb{H}''* \mathbb{H}'}$ is an interval and  $\iota h(0)=f(a)$ and $\iota h(1)=b$. The element $b$ was arbitrarily chosen, so this proves that $f$ is essentially surjective. 

\end{proof}

\begin{cor}\label{ess.surj.cor}
Let $\Mm$ be a monoidal model category which admits transfer for operads. In $\EOper{\Mm}$ the transfinite composition of essentially surjective morphisms is essentially surjective.
\end{cor}
\begin{proof}
 It follows directly from the definition of essential surjectivity in $\EOper{\Mm}$ and Lemma \ref{ess.surj.lemma}.
\end{proof}

The lemma that follows is the enriched version of \cite[Lemma 1.29]{CM11}:

\begin{lemma}\label{lemma.push2}
Consider a full embedding $\morp{i}{A}{B}$ in $\ECat{\Mm}$, in which $A$ has one object $0$ and $B$ has as set of objects $\{ 0,1 \}$ such that $i(0)=0$.
For every push-out in $\EOper{\Mm}$ along $j_!(i)$:
\[
 \diagc{j_!(A)}{P}{j_!(B)}{Q}{j_!(i)}{a}{v}{b}{}
\]
the morphism of operads $v$ is fully faithful.
\end{lemma}
\begin{proof} This is a particular case of Proposition \ref{prop.ap.ff}.

\end{proof}

\begin{lemma}\label{cell1}
 Suppose $\Mm$ is compactly generated with cofibrant monoidal unit and satisfies the monoid axiom, then every relative $\mathcal{I}$-cell in $\ENSOper{\Mm}$ ($\EROper{\Mm}$) is a local $\otimes$-cofibration.
\end{lemma}
\begin{proof}
As stated in Proposition \ref{compstone}, $\Mm$ admits transfer for reduced operads and non-symmetric operads. Since local $\otimes$-cofibrations are closed under transfinite composition
it is sufficient to show that in any push-out square
\begin{equation}\label{dgc}
 \diagc{A}{X}{B}{Y}{i}{f}{j}{g}{}
\end{equation}
if $i$ is in $\mathcal{I}$ then $j$ is a local $\otimes$-cofibration.
Since the unit is supposed to be cofibrant it is immediate to check that a push-out of $\arr{\emptyset}{j^*(\obCat)}$ is a local cofibration (hence a local $\otimes$-cofibration).
Suppose now that $i\in I_{loc}$, i.e. $i=C_n(k)$ for some $\morp{k}{a}{b}$ in $I$ (where $I$ is the set of generating cofibrations of $\Mm$); then the statement follows from Proposition \ref{compstone.alg}. In fact, in this case the push-out (\ref{dgc}) can be obtained from the following push-out diagram in $\ENSOperfc{\Mm}{C}$ (resp. $\EROperfc{\Mm}{C}$) where $C=\Cl{X}$ 
\[
 \diagc{f_!C_n(a)}{X}{f_!C_n(b)}{Y}{f_!(i)}{\trc{f}}{j}{\trc{g}}{}
\]
(see Appendix \ref{sez.colim} for a description of colimits in these categories).
Since $f_!(i)$ is a cofibration in $\ENSOperfc{\Mm}{C}$ (resp. $\EROperfc{\Mm}{C}$), $g$ is also a cofibration in the same category, hence a local $\otimes$-cofibration from Proposition \ref{compstone.alg}.  
\end{proof}

\begin{lemma}\label{cell2}
Suppose $\Mm$ is a cofibrantly generated monoidal category with cofibrant unit:
\begin{enumerate}
 \item\label{cell2-1} if $\Mm$ admits transfer for operads, a set of generating $\Mm$-intervals $\mathfrak{G}$ and the class of weak equivalences is closed
under transfinite composition, then every relative $\mathcal{J}$-cell in $\EOper{\Mm}$ is a weak equivalence.
\item\label{cell2-1b} if $\Mm$ is $K$-compactly generated, admits $K$-transfer for operads and a set of generating $\Mm$-intervals $\mathfrak{G}$
, then every relative $\mathcal{J}$-cell in $\EOper{\Mm}$ is a weak equivalence which and a local $K$-morphism (Definition \ref{d.locK} and \ref{defmod}).
\item \label{cell2-2} if $\Mm$ is compactly generated, satisfies the monoid axiom and has a set of generating $\Mm$-intervals $\mathfrak{G}$, then every relative 
$\mathcal{J}$-cell in $\ENSOper{\Mm}$ and $\EROper{\Mm}$ is an essentially surjective local trivial $\otimes$-cofibration (in particular it is a weak equivalence). 
\end{enumerate}
\end{lemma}
\begin{proof}
The proof is similar to the one of Lemma \ref{cell1}. 
By Lemma \ref{ess.surj.lemma} essentially surjective maps are closed under transfinite composition.
 Under the hypothesis of point (\ref{cell2-1}) local weak equivalences are closed under transfinite composition so we are reduced to prove that, for every push-out diagram like (\ref{dgc}), if $i\in \mathcal{J}$ then $j$ is a weak equivalence.
If $i\in J_{loc}$ then this follows from the fact that the push-out can be calculated in $\EOperfc{\Mm}{\Cl{X}}$; this implies that $\trf{j}$ is a trivial cofibration in $\EOperfc{\Mm}{\Cl{X}}$, thus $j$ is local weak equivalence bijective on colours (in $\EOper{\Mm}$), hence a weak equivalence.

Let us consider now the case in which $\morp{i=j_!(i_1)}{j_!(\obCat)}{j_!(\mathbb{G})}$ for some $\mathbb{G}\in \mathfrak{G}$:
\[
 \diagc{j_!(\obCat)}{X}{j_!(\mathbb{G})}{Y\, .}{i}{j}{}{}{}
\]
We can decompose it in two push-outs
\[
 \xymatrix{j_!(\obCat) \ar[d]^{j_!(\phi)} \ar[r]^{p}   & X \ar[d]^{\phi'} \\
           j_!(\mathbb{G}_{0,0}) \ar[r] \ar[d]^{j_!(\psi)} & X'\ar[d]^{\psi'}\\
           j_!(\mathbb{G}) \ar[r]                         & Y}
\]
where $\mathbb{G}_{0,0}$ is the full subcategory of $\mathbb{G}$ which has $0$ as unique object, and $\psi$ is the canonical inclusion.
The functor $\psi$ is fully faithful, hence Lemma \ref{lemma.push2} implies that $\psi'$ is fully faithful (hence a local weak equivalence). Since $\psi'$ is essentially surjective by construction, it is a weak equivalence in $\EOper{\Mm}$.

Let us consider now the  push-out on top. It can be calculated in $\EOperfc{\Mm}{\Cl{X}}$ (see Appendix \ref{sez.colim})
as the push-out of $\trc{p}$ (which is a morphism from $p_!(j_!(\obCat))$ to $X$) along $p_!(j_!(\phi))$, in particular $\phi'$ is bijective on colours (so it is essentially surjective).

Under our hypotheses $\phi$ is a cofibration in $\EOperfc{\Mm}{\singSet}$ (according to the \cite[Interval Cofibrancy Theorem 1.15]{BM12}), furthermore it is a trivial cofibration
since $\mathbb{G}$ is weakly equivalent to $\obCat$.

Note that  $p_!(j_!(\phi))=p_!(j_{\singSet !}(\phi))$ and $p_!$ and $j_{\singSet !}$ are left Quillen functors, so $p_!(j_!(\phi))$ is a trivial cofibration in 
$\EOperfc{\Mm}{\Cl{X}}$. This implies that $\phi'$ is a trivial cofibration in $\EOperfc{\Mm}{\Cl{X}}$ as well; hence a local weak equivalence in $\EOper{\Mm}$.

 Since $\phi'$ is furthermore bijective on colours, it is a weak
equivalence in $\EOper{\Mm}$ and this concludes the proof of point (\ref{cell2-1}). 

The proof of point (\ref{cell2-1b}) and (\ref{cell2-2}) are almost identical, keeping in mind Proposition \ref{compstone.alg} and the fact that in a $K$-compactly generated model category the class of maps which are local weak equivalence and local $K$-morphisms is closed under transfinite composition. 
\end{proof}

\begin{lemma}\label{small.lemma2}
Let $\Mm$ be a cocomplete monoidal closed category and $n\in \N$ then:
\begin{enumerate}
 \item If $X\in \Mm$ is small then $C_n(X)$ is small in $\EOper{\Mm}$.
 \item Suppose $\Mm$ is a $K$-compactly generated model category. If $X\in \Mm$ is $K$-small then $C_n(X)$ is small in $\ENSOper{\Mm}$ (resp. $\EROper{\Mm}$) with respect to local $K$-morphisms.
\end{enumerate}
\end{lemma}
\begin{proof}
Recall that $C_n(X)=F_{\Opop{[n]}} \iota_{s_n}(X)$. Both statements are a straight-forward consequence of Lemma \ref{small.lemma} and Lemma \ref{lemma.adj}. 
\end{proof}
 
\subsection{Model structure}

We are now ready to prove that there exists a cofibrantly generated model structure on $\EOper{\Mm}$ with (\ref{eq.gen.c}) and (\ref{eq.gen.tc})  as generating cofibrations and trivial cofibrations and whose class of weak equivalences $\mathcal{W}$ is the one defined in Definition \ref{defmod}.

This is our main theorem:

\begin{thm}\label{mod.str.thm}
Let $\Mm$ be a monoidal model category  such that:
\begin{itemize}
\item The unit is cofibrant;
\item The model structure is right proper;
\item There exists a set $\mathfrak{G}$ of generating $\Mm$-intervals
\end{itemize}
Then:
\begin{enumerate}
\item\label{mod.str.thm.p1} If $\Mm$ admits transfer for operads (Definition \ref{opadefi}) and the class of weak equivalences is closed under transfinite composition,
then there exists a cofibrantly generated model structure on $\EOper{\Mm}$ such that fibrations and weak equivalences are the ones introduced in 
Definition \ref{defmod}.

\item\label{mod.str.thm.p1b} If $\Mm$ is $K$-compactly generated, admits $K$-transfer for operads (Definition \ref{opadefi}) 
, then there exists a cofibrantly generated model structure on $\EOper{\Mm}$ such that fibrations and weak equivalences are the ones introduced in 
Definition \ref{defmod}.

\item\label{mod.str.thm.p2} If $\Mm$ is compactly generated, satisfies the monoid axiom and admits transfer for non-symmetric operads (reduced operads) then there exists a cofibrantly generated model
structure on  $\ENSOper{\Mm}$ (resp. $\EROper{\Mm}$) such that fibrations ans weak equivalences are as in Definition \ref{defmod}.
\end{enumerate}
\end{thm}
\begin{proof}
We pick (\ref{eq.gen.c}) as set of generating cofibrations and (\ref{eq.gen.tc}) as set of generating trivial cofibrations.

Since the category of $\EOper{\Mm}$, $\ENSOper{\Mm}$ and $\EROper{\Mm}$  are complete and cocomplete, 
to prove the existence of the model structure  it is sufficient to prove (following Theorem 2.1.19 \cite{Ho99})) that:
\begin{enumerate}
 \item the class of weak equivalences has the 2-out-of-3 property and it is closed under retracts;
\item the domains of $\mathcal{I}$ are small relative to $\cell{\mathcal{I}}$;
\item the domains of $\mathcal{J}$ are small relative to $\cell{\mathcal{J}}$;
\item $\inj{\mathcal{I}}= \mathcal{W}\cap \inj{\mathcal{J}}$
\item $\cell{\mathcal{J}}\subseteq \mathcal{W}\cap \cof{\mathcal{I}}$.
\end{enumerate}

(1) holds since $\mathcal{W}$ has the 2-out-of-3 property by Proposition \ref{prop2.12} and it is closed under retract because local weak equivalence are closed
under retracts  and essentially surjectivity is closed under retracts (because it is so in $\ECat{\Mm}$, see proof of  \cite[Theorem 2.5]{BM12}). 
Point 4 is Lemma \ref{lemma2.4}.
Points 2 and 3 follow from Lemma \ref{small.lemma2}.
Point 5 is Lemma \ref{cell2}.
\end{proof}

\begin{exm} 
Model categories for which the first part of the theorem applies are, for example: simplicial sets with the Quillen model
structure, (unbounded) chain complexes over a filed of characteristic 0 with the projective model structure,
simplicial modules over any ring (with the model structure transferred from the Quillen's one on simplicial sets)
and $\Set$ with the model structure which has bijection as weak equivalences and surjections as fibrations.
In the case of simplicial sets this was already proven in \cite[Theorem 1.14]{CM11} and \cite{Ro11}. For $\Mm=\Set$
on recover the ``folk'' model structure on Operads (see \cite{We07}).

The category $\Top$ of (weak Hausdorff) compactly generated spaces satisfies the hypotheses of point (\ref{mod.str.thm.p1b}) if we take $K=\mathcal{T}_1$, the class of $T_1$-closed inclusions. 

All the categories listed above also satisfy the hypothesis of point (\ref{mod.str.thm.p2}). An example of model category which satisfy (\ref{mod.str.thm.p2}) but not (\ref{mod.str.thm.p1}) is the category of chain complexes over an arbitrary commutative ring with 
the projective model structure. 

\end{exm}

\subsection{Remark on weak equivalences}

In \cite{Ro11} weak equivalences are defined in a different way which is not a priori equivalent to ours:

\begin{defi}
A morphism $\morp{f}{P}{Q}$ of $\Mm$-enriched coloured operads is a \emph{Dwyer-Kan weak equivalence} if:
\begin{enumerate}
\item is a locally weak equivalence, i.e. for every $s\in \Sqofc{\Cl{P}}$ the morphism
\[
\morp{\trf{f}}{P(s)}{f^*Q(s)}
\]
is a weak equivalence in $\Mm$.
\item The induced functor
\[
\morp{\homcc{j^*f}}{\homcc{j^*P}}{\homcc{j^*Q}}
\]
is essentially surjective. 
\end{enumerate}
\end{defi} 
We recall that for $A\in \ECat{\Mm}$ the ordinary category $\homcc{A}\in \Cat$ is the category which has as objects the objects of $A$ and for every $x,y\in A$ set of morphisms $\homcc{A}(x,y)=\homcat{\Mm}(\Un_{\Mm}, A(x,y))$; the composition in $\homcc{A}(x,y)$ is defined in the evident way (\emph{cf.} \cite[Remark 2.7]{BM12}).

It was already observed by Berger and Moerdijk that under our hypotheses (so also in the case of \cite{Ro11} where $\Mm$ is simplicial sets with the Kan-Quillen model structure) the two definitions coincide.
  
The following proposition is a straight-forward consequence  of \cite[Propositions 2.20,2.24]{BM12}:
\begin{prop}
Suppose that $\Mm$ is a monoidal model category with a cofibrant unit admitting transfer for categories. Then the class of morphisms in $\EOper{\Mm}$ which are essentially surjective locally weak equivalences coincides with the class of Dwyer-Kan weak equivalences.   
\end{prop}

\section{Right properness}

We recall the following lemma, whose proof is trivial:
\begin{lemma}
 Let $\Mm$ be a right proper monoidal model category and $\mathcal{O}$ a coloured $\Mm$-operad. If the category of $\mathcal{O}$-algebras in $\Mm$ admits a transferred model structure then it is right proper.  
\end{lemma}

Taking $\mathcal{O}=\Opop{C}$ the previous lemma reads
\begin{cor}\label{rp.alg}
If $\Mm$ is right proper and admits transfer for operads (non-symmetric operads, reduced operads) then, for every $C\in \Set$ the transferred model structure on $\EOperfc{\Mm}{C}$ ($\ENSOperfc{\Mm}{C}$, $\EROperfc{\Mm}{C}$) is right proper. 
\end{cor}

We can now prove the following:
 
\begin{prop}
Suppose that $\Mm$ satisfies the hypothesis of point (\ref{mod.str.thm.p1}) of Theorem \ref{mod.str.thm}. Then the model structure on $\EOper{\Mm}$ is right proper. 
\end{prop}
\begin{proof}
Suppose a pullback diagram is given in $\EOper{\Mm}$
\[
\diagc{A}{X}{B}{Y}{f'}{w'}{f}{w}{}
\]
where $f$ is a fibration and $w$ is a weak equivalence. We have to prove that $w'$ is a weak equivalence, i.e. it is an essentially surjective local weak equivalence.

Note that, the diagram
\[
\diagc{A}{{w'}^*(X)}{f'^*(B)}{(fw)^*(Y)}{\trf{f'}}{\trf{w'}}{w'^*(\trf{f})}{f'^*(\trf{w})}{}
\]
is a pullback diagram in $\EOperfc{\Mm}{\Cl{A}}$, which is right proper by \ref{rp.alg} , hence $w'$ is a local weak equivalence.

We only have to prove that $w'$ is essentially surjective. Suppose a colour $x$ of $X$ is given. Since $w$ is essentially surjective we have an object $b$ in $B$, an interval $\mathbb{H}$, a morphism $\morp{h}{\mathbb{H}}{Y}$, such that the following diagram commutes
\[
\xymatrix{\mathbf{j_!(I)}\ar[dr]_{j_!(i_0)}\ar@/_1pc/[ddd]_{b}\ar@{-->}[dd]^{a} & & \mathbf{j_!(I)}\ar[dl]^{j_!(i_1)}\ar[dd]^{x}\\
& j_!(\mathbb{H})\ar[ddr]^{h}\ar@{-->}[dr]^{h'}\\
A\ar[rr]^{w'}\ar[d]^{f}& &X\ar[d]^{f}\\
B\ar[rr]^{w}& &Y
}
\]
Since $j_!(i_1)$ is a trivial cofibration and $f$ is a fibration, $h$ lifts to a morphism $\morp{h'}{j_!(\mathbb{H})}{X}$ such that $h'(1)=x$ and $fh'(0)=w(b)$; from the universal property of $A$ there is a colour $a\in \Cl{A}$,  such that $w'(a)=h'(0)$. Since $x$ was chosen arbitrarily, this proves that $w'$ is essentially surjective.  
\end{proof}

In an analogous way we can prove the same holds for the non-symmetric and the reduced case:

\begin{prop}
Suppose that $\Mm$ satisfies the hypothesis of point (\ref{mod.str.thm.p2}) of Theorem \ref{mod.str.thm}. Then the model structures on $\ENSOper{\Mm}$ and $\EROper{\Mm}$ are right proper. 
\end{prop}

\appendix

\section{Filtered colimits in $\EOper{\Mm}$}\label{sez.colim}
\subsection{Colimits in bifibrations}
Suppose we have a bifibration $\morp{\pi}{\mathcal{C}}{\Set}$ such that for every $C\in\Set$ the fiber $\Fib{\pi}{C}$ is cocomplete. We want to show that 
$\mathcal{C}$ is cocomplete as well, expressing the colimits in $\mathcal{C}$ as colimits in certain fibers.

Fix a small category $I$. Let $\Delta [1]$ be the category representing morphisms, i.e. the category with two object $0$ and $1$ and only one morphism from $0$ to $1$. 

Let $\mathcal{D}$ be any category. A functor $\morp{F}{I\times \Delta [1]}{\mathcal{D}}$ is just a natural transformation from $F_{| I\times \{0\}}$ to $F_{| I\times \{1\}}$. Given $i\in I$ we will denote with $F_i$ the morphism in $\mathcal{D}$ given by $F_{\{i\}\times \Delta[1]}$ (in this way we recover the usual notation for natural transformations).

 Given two natural transformations $\morp{F,G}{I\times \Delta [1]}{\mathcal{C}}$ such that $F_{| I\times \{1\}}=G_{| I\times \{0\}}$, we denote by $\morp{G\ast F}{I\times \Delta [1]}{\mathcal{C}}$ their (horizontal) composite.

Given a morphism $\morp{f}{a}{b}$ in $\mathcal{C}$ (that we can see as a functor $\morp{f}{\Delta[1]}{C}$) let $\morp{\bar{f}}{I\times \Delta [1]}$ be the constant natural transformation from $a$ to $b$, i.e. $\bar{f}= f\circ p_2$ where $\morp{p_2}{I\times \Delta[1]}{\Delta[1]}$ is the second projection; in other words $\bar{f}_i=f$ for every $i\in I$.

Suppose we want to compute the colimit of a certain functor $\morp{L}{I}{\mathcal{C}}$. We can first compute the colimit of $\pi L$ in $\Set$, that we can see as a functor from $\morp{l}{I\times \Delta[1]}{\Set}$ such that $l_{|I\times \{0\}}=\pi L$ and $l_{|I\times \{1\}}$ is constant of value $l_1$. 
We can lift $l$ to a ``cocartesian cylinder`` $\morp{c}{I\times \Delta[1]}{\mathcal{C}}$, this means that  $\pi c=l$ and $c_{|I\times \{0\}}=L$, and for every other $\morp{k}{I\times \Delta[1]}{\mathcal{C}}$ such that $\pi k=f \ast l$ and $k_{|I\times \{0\}}=L$, there exist a unique $\morp{f'}{I\times \Delta[1]}{\mathcal{C}}$ such that $f'_{|I\times \{0\}}=c_{|I\times \{1\}}$, $f'_{|I\times \{1\}}=k_{|I\times \{1\}}$, its projection on sets is $\pi(f')=f$ and $f'\ast c=k$.

Note that the image of $c_{|I\times \{1\}}$ lies in $\Fib{\pi}{l_1}$; we claim that the colimit of $L$ is $q \ast c$ where $\morp{q}{I\times \Delta[1]}{\Fib{\pi}{l_1}}$ is the colimit of $c_{|I\times \{1\}}$.

In fact, take another cocone $\morp{k}{I\times \Delta[1]}{\mathcal{C}}$ such that $k_{|I\times \{0\}}=L$ and $k_{|I\times \{1\}}$ is constant of value $k_1$. By the universal property of $l$ (it is the colimit of $\pi L$) the functor $\pi k$ factors in a unique way as $(\pi k)\ast \bar{z}=l$ where $\morp{z}{l_1}{\pi(k_1)}$ is the morphism given by the universal property of $l_1$.

The cocone $\bar{z}$ can be lifted in a unique way to a cocone $\morp{f'}{I\times \Delta[1]}{\mathcal{C}}$ from $c_{|I\times \{1\}}$ to $k_1$ such that $f'\ast c=k$. Since $\pi$ is a fibration $f'$ factors in a unique way as $f'=\bar{\phi_{z}} q'$ where $q'$ is a cocone in $\Fib{\pi}{l_1}$ and $\morp{\phi_{z}}{z^*(k_1)}{k_1}$ is the map of a chosen cleavage. By the universal property of $q$ there is unique map $\morp{w}{q_1}{z^*(k_1)}$ in $\Fib{\pi}{l_1}$ such that $\bar{w}\ast q=q'$. It is easy to see that $\overline{\phi_z w}\ast (q \ast c)=k$ and $\phi_z w$ is the unique map with this property.

Summarizing, to compute $\varinjlim L$ we first compute $\varinjlim (\pi L)=l_1$; We consider the ''direct image`` of $L$ in $\Fib{\pi}{l_1}$:
\[
 \gfun{L_!}{I}{\Fib{\pi}{l_1}}{i}{l_i!(L(i))}
\]
(this is nothing but $c_{|I\times \{1\}}$ in the previous description).
What we have proved is that every cocone from $L$ to an object $k_1$ of $C$ factor in a unique way through the ''cocartesian cylinder'' and a cone from $L_!$ to $k_1$, and every cone from $L_!$ to $k_1$ factors in a unique way trough $\varinjlim (L_!)$. Hence $\varinjlim L \simeq \varinjlim (L_!)$.

If $\Mm$ is a cocomplete symmetric monoidal category then for every coloured operad $\mathcal{O}$ the category of algebras $\Alg{\mathcal{O}}{\Mm}$ is cocomplete. This implies in particular that the fibers of the bifibred categories $\ECat{\Mm}$, $\EOper{\Mm}$, $\ENSOper{\Mm}$, $\EROper{\Mm}$ are cocomplete. The fibers of $\EMGraph{\Mm}$,$\EMGraph{\Mm}$ and $\ERMGraph{\Mm}$ are also cocomplete; hence we can use the previous method for the description of colimits in these categories (and obtain another proof that they are all cocomplete).

\subsection{Filtered colimits in $\EGraph{\Mm}$ and $\EOper{\Mm}$} 
Even though the description of the colimits given in the previous section is very general, it might not be very effective for actual computations of colimits.

In the case in which the bifibration we consider is $\morp{\pi}{\EMGraph{\Mm}}{\Set}$ and $I$ is filtered a more explicit description can be given. 
Let $\morp{L}{I}{\EMGraph{\Mm}}$; let $L^{\sharp}=\pi L$ and let $\morp{l}{I\times\Delta[1]}{\Set}$ be its colimit cocone with vertex $l_1\in Set$. 

Every signature $s\in \Sqofc{l_1}$ is of the form $(l_i(d))$ for some $i\in I$ and $d\in \Sqofc{L^{\sharp}(i)}$. 
The couple $(i,d)$ will be called a \emph{representative} for $s$.
The colimit of $L$ is a multi-graph $C$ with set of colours $l_1$. For every $s\in \Sqofc{l_1}$ the $s$-component of $C$ is:
\begin{equation}\label{eq.filt.comp}
 C(s)=\varinjlim_{(j,f_j)\in i/I} L_j(f_j(d))
\end{equation}
where $(i,d)$ is some representative for $s$ (it can be shown that this definition does not depend on the choice of the representatives).

There is an obvious cocone  $\morp{\eta}{I\times \Delta[1]}{\EGraph{\Mm}}$ from $L$ to $C$ such that $\eta=\pi l$. It is easy to check that $\eta$ is a colimit for $L$.

Note that for every $n\in \N$ and every $s,s_1,\dots, s_n\in \Sqofc{c}$ for which it make sense to write $s\circ(s_1,\dots,s_n)$, one can find an $i\in I$ and $d,d_1,\dots,d_n\in \Sqofc{L^{\sharp}(i)}$ such that 
\begin{itemize}
 \item $(i,d)$ is a representative of $s$;
 \item for every $j\in [n]$ the couple $(i,d_j)$ is a representative of $s_j$;
 \item it makes sense to write $d\circ(d_1,\dots,d_n)$;
 \item $(i,d\circ(d_1,\dots,d_n))$ is a representative for $s\circ(s_1,\dots,s_n)$.
\end{itemize}

If $s$ is of the form $(a;a)$ for some $a\in c$ then $d$ can be chosen of the form $(v;v)$ for some $v\in \pi(L(i))$.

We remark also that given $\sigma\in \Sigma_{\card{s}}$ if $(i,d)$ is a representative for $s$, then $(i,\sigma d)$ is a representative for $\sigma s$.

The same description of filtered colimits can be given in $\ERMGraph{\Mm}$.

We want to show that the right adjoint functors in the adjunctions (\ref{eq.ff.op2}),(\ref{eq.ff.nsop2}) and (\ref{eq.ff.rop2}) create filtered colimits.

\begin{lemma}\label{lemma.adj}
Let $\Mm$ be a cocomplete monoidal closed category. 
\begin{enumerate}
\item\label{unique.case} The right adjoint functor $U_{\Opop{}}$ between $\EMGraph{\Mm}$ and $\EOper{\Mm}$  creates
filtered colimits.
 \item The right adjoint functor $U_{\Opnsop{}}$ between $\EMGraph{\Mm}$ and $\EOper{\Mm}$ creates
filtered colimits.
 \item The right adjoint functor $U_{\Oprop{}}$ between $\ERMGraph{\Mm}$ and $\EROper{\Mm}$ creates
filtered colimits.
\end{enumerate}
\end{lemma}
\begin{proof}
We are going to prove point \ref{unique.case}  only, the other proofs are almost identical.

 Let $I$ be a small filtered category and let $\morp{L'}{I}{\EOper{\Mm}}$. If $L=U_{\Opop{}} L'$ and $\morp{\eta}{I\times\Delta[1]}{\EMGraph{\Mm}}$ be colimit cocone of $L$ with apex $C$ described as above. It will be convenient to denote $L(i)$ by $L_i$.

we have to prove that $C$ can be endowed of an operad structure such that for every $i\in I$ $\morp{n_i}{L_i}{C}$ becomes a morphism of operads.

The structure of operad on $C$ is induced in a natural way from the ones on the $L_i$'s thanks to the fact that $\Mm$ is monoidal closed.

In fact, for every $n\in \N$ and every $s,s_1,\dots, s_n\in \Sqofc{c}$ for which it makes sense to write $s\circ(s_1,\dots,s_n)$, we can find 
$i\in I$ and $d,d_1,\dots,d_n\in \Sqofc{\pi(L_i)}$ such that
\[ 
C(s)=\varinjlim_{(j,f_j)\in i/I} L_j(f_j(d)) \ 
\]
\[
C(s_k)=\varinjlim_{(j,f_j)\in i/I} L_j(f_j(d_k)) \ 
\]
for every $k\in [n]$, and such that to write $d\circ(d_1,\dots,d_n)$ makes sense.

Since $i/I$ is filtered and $\Mm$ is monoidal closed:
\[
 (\underset{k\in [n]} \bigotimes C(s_k))\otimes C(s) \simeq \varinjlim_{(j,f_j)\in i/I} (\underset{k\in [n]} \bigotimes L_j(f_j(d_k)))\otimes L_j(f_j(d))
\]

The composition $\Gamma$ is defined as the unique map that makes the following diagram commute:
\[
\diagc{(\underset{k\in [n]} \bigotimes L_j(f_j(d_k)))\otimes L_j(f_j(d))}
      {L_j(f_j(d\circ(d_1,\dots,d_n)))}
      {(\underset{k\in [n]} \bigotimes C(s_k))\otimes C(s)}
      {C(d\circ(d_1,\dots,d_n))}{}{\Gamma_j}{}{\Gamma}{}
\]

where the vertical maps are the canonical ones defined by the colimits.

The morphism $\Gamma$ exists (and is unique) because the following diagram commutes for every morphism $\morp{g}{(j,f_j)}{(z,f_z)}$ in $i/I$:
\[
\diagc{(\underset{k\in [n]} \bigotimes L_j(f_j(d_k)))\otimes L_j(f_j(d))}
      {L_j(f_j(d\circ(d_1,\dots,d_n)))}
      {(\underset{k\in [n]} \bigotimes L_z(f_z(d_k)))\otimes L_z(f_j(d))}
      {L_z(f_z(d\circ(d_1,\dots,d_n)))}
      {\underset{k\in [n]} \bigotimes L(g)\otimes L(g)}{\Gamma_j}{ L(g)}{\Gamma_z}{}
\]
since all $L(g)$'s are morphisms of operads.

Let $s\in \Sqofc{c}$ and $(i,d)$ one of its representatives. For every $\sigma\in \Sigma_{\card{s}}$
the right action of $\sigma$ is defined as the unique map $\sigma^*$ (on the second row) that makes the following diagram commute for every $(j,f_j)\in i/I$:
\[
\diagc{L_j(f_j(d))}{L_j(f_j(d)\sigma)}{C(s)}{C(s\sigma)}{}{\sigma^*}{}{\sigma^*}{}
\]
where the $\sigma^*$ on the first row is given by the action of $\sigma$ on $L_j$ and the vertical arrows are the canonical arrows of the colimits defining $C(s)$ and $C(s\sigma)$.

The only things left to define are the identity operations. Given a colour $a \in c$ let $(i,(v;v))$ be a representative for $(c;c)$. The identity operation for $a$ is defined as the unique map $u_a$ such that the following diagram commutes for every $(j,f_j)\in i/I$:
\[
\xymatrix{& L_j(f_j(v;v)) \ar[d]\\
          I \ar[ur]^{u_{f_j(v)}} \ar[r]^{u_a} & C(a;a)}
\]
where $u_{f_j(v)}$ is the identity for $f_j(v)$ in $L_j$ and the vertical arrow is the one that comes from the definition of $C(a;a)$ as colimit.

It is left to the reader to verify that this is a good definition of operad structure and that, with this structure, all $\eta_i$'s become morphisms of operads. 

\end{proof}

The explicit description of filtered colimits of multi-graphs allows to prove the following:

\begin{lemma}\label{small.lemma}
Let $\Mm$ be a cocomplete monoidal closed category.
If $X\in \Mm$ is small respect to a certain class of maps $\mathcal{S}$ then,  $\iota_{s_n}(X)$ (see Section \ref{sec.adm.op}) is small in $\EMGraph{\Mm}$ (in $\ERMGraph{\Mm}$) respect to $\mathcal{S}$-local morphisms, for every $n\in \N$. 
\end{lemma}
\begin{proof}
Let $\omega$ be an ordinal number and $\morp{L}{\omega}{\EMGraph{\Mm}}$ a continuous functor.
Suppose a morphism $\morp{f}{\iota_{s_n}(X)}{\varinjlim{L}}$ is given. We have to prove that $f$ factors through some $L_j=L(j)$ for some $j\in \omega$.
This morphism is completely determined by the induced map on the colours $\morp{\Cl{f}}{[n+1]}{\Cl{\varinjlim{L}}}$ and its evaluation at the
$s_n$-level:
\begin{equation}\label{eq.small}
 \morp{f}{\iota_{s_n}(X)(s_n)=X}{\varinjlim{L}(\Cl{f}(s_n))}.
\end{equation}
Since $[n+1]$ is finite, $\Cl{f}$ factors through some $\morp{h}{[n+1]}{\Cl{L(i)}}$ for some $i\in \omega$. The couple $(i,h(s_n))$ is then a representative for $f^{\sharp}(s_n)$, and the codomain of (\ref{eq.small}) can be rewritten as $C(s)=\varinjlim_{(j,g_j)\in i/\omega} L_j(g_j(h(s_n)))$ (according to formula \ref{eq.filt.comp}), which is a transfinite composition of morphism in $\mathcal{S}$.

Since $X$ is small respect to $\mathcal{S}$, the morphism (\ref{eq.small}) factors through a morphism $\morp{f'}{X}{L_j(g_j(h(s_n))}$ for some $(j,g_j)\in i/\omega$. This morphism together with the map on the colours $\morp{g_j h}{[n+1]}{\Cl{L_j}}$ determines a morphism of multi-graphs $\morp{f'}{\iota_{s_n}(X)}{L_j}$ such that $f=k_j f'$, where $\morp{k_j}{L_j}{\varinjlim{L}}$ is the canonical map associated to the colimit.

The proof for the case of reduced multi-graphs is identical.
\end{proof}

\section{Push-out of $C$-coloured Operads}\label{App.B}

The aim of this appendix is to give a proof of Proposition \ref{prop.ap.ff}. An explicit proof of this proposition will be given only in the case of $\EOper{\Mm}$. The main ingredient of this proof is the explicit description of push-outs in $\EOperfc{\Mm}{C}$ given in Proposition \ref{exp.push}. A similar description can be given for push-outs in $\ENSOperfc{\Mm}{C}$, using planar trees instead of non-planar ones. Once such a description is given, the proof of Proposition \ref{prop.ap.ff} in the case of $\ENSOper{\Mm}$ becomes almost identical to the one for symmetric operads. Since the push-outs of reduced operads are still reduced operads, the case of $\EROper{\Mm}$ follows from the one of $\EOper{\Mm}$.

\subsection{Trees}\label{sec1}
Let $C$ be a set and $(\Mm,\otimes,\mathbb{I})$ a cocomplete monoidal category.
We want to give an explicit description of the push-outs in $\EOperfc{\Mm}{C}$.
 
Let $\Gamma$ be the category with three objects $o,x,y$, generated by the following graph:
\[
\xymatrix{o\ar[d]^{\iota_2} \ar[r]^{\iota_1} & x\\
          y & 
}.
\]

Given a category $\mathcal{C}$, a push-out diagram is just a $\Gamma$-diagram $\morp{F}{\Gamma}{\mathcal{C}}$ together with a colimit for it.
\begin{equation}\label{pushdiag}
 \diagc{F(o)}{F(x)}{F(y)}{Z}{F(\iota_1)}{F(\iota_2)}{p}{q}{}.
\end{equation}

For a description of this colimit in $\EOperfc{\Mm}{C}$ we need to introduce trees.
We are going to define for every $S\in \Sqofc{C}$ a category of trees (with a certain marking on vertices) $\trees{C}_S$ which will be helpful
to define $Z(S)$; more precisely $Z(S)$ will be a colimit indexed by $\trees{C}_S$.

Talking about operads the natural "category of trees" to work with is the dendroidal category $\Omega$ introduced by Moerdijk and Weiss (\cite{MW07},\cite{We07}). 

We are now going to briefly recall the definitions of $\Omega$ and other categories derived from it.

\subsubsection{The dendroidal category}

Following Weiss exposition a \emph{graph} will be for us a couple $(E,V)$ where $E$ is a non-empty set, called the set of \emph{edges}, and $V\subset P(E)$  such that every $e\in E$ belongs to at most two element of $V$; The set $V$ is called the set of \emph{vertices}. Given a graph $G$ the set of vertices will be denoted $\verx{G}$ and the set of edges will be indicated with $\edge{G}$.

The edges that belong to only one vertex are called \emph{outer edges} and the others are \emph{inner edges}. 

Given a graph $G=(E,V)$, a \emph{subgraph} of $G$ is a graph $(E',V')$ such that $E'\subset E$ and $V'\subset V$. A \emph{path} of length $n$ in $G$ is a f sequence
$(e_1,\dot, e_n)$ of distinct element of $E$ and a sequence $(v_1,\dots,v_{n-1})$ of distinct element of $V$ such that $\{e_i,e_{i+1}\}\subset v_i$ for every
$1\leq i <n$; the edges $e_1$ and $e_n$ are called 
the \emph{extremes of the path}. A \emph{loop} in $G$ is a path of length greater then $1$ with the extremes belonging to a common vertex.
Two edges $e,f$ of $G$ are \emph{connected} if they are extremes of a common path; this path, if it exists, can be chosen of minimal length, this minimal length is called the \emph{distance} between $e$ and $f$. A graph is \emph{connected} if every pair of edges is \emph{connected}.

A \emph{tree} is a finite connected graph with no loops and a chosen outer edge called the \emph{root}. The outer edges different from the root are called \emph{leaves}. The root of a tree $T$ will be denoted by $\troot{T}$ and the set of leaves will be called $\leav{T}$. 

A \emph{subtree} of a tree $T$ is just a subgraph of $T$ equipped with the structure of a tree whose root is the outer edge with minimal distance from $\troot{T}$.
 
Trees with only one vertex are called \emph{corollas}; in a corolla the unique vertex is always the whole set of edges. The tree with one edge and no vertices will be called the \emph{empty tree} and denoted by $|$.

The \emph{arity} of a non-empty tree $T$ is defined as the corolla $\arity{T}$ whose set of edges is the set of outer edges of $T$ and the root is $\troot{T}$.

For every vertex $v$ of a tree $T$ we will denote by $t_v$ the corolla which has $v$ as set of edges and as root the edge connected to the root of $T$ with the shortest path (one can check that this is a good definition). The set $\leav{t_v}$ is also denoted by $\leav{v}$.

Every tree $T$ freely generates a coloured operad (in $\Set$) $\Omega(T)$ whose set of colours is $\edge{T}$. We will not give a detailed description of $\Omega(T)$, but we want to recall that there is a bijective correspondence between $\underset{S\in \Sqofc{\edge{T}}}\coprod \Omega(T)(S)$ (the set of operations of $\Omega(T)$) and the couples $(t,\tau)$ where $t$ is a subtree of $T$ and $\tau$ is a total order on $\leav{t}$ and each $\Omega(T)(S)$ has at most one element.

The objects of $\Omega$ are trees and a morphism between $T$ and $T'$ is a morphism of operads between  $\Omega(T)$ and $\Omega(T')$.

Given a map $\morp{f}{T}{T'}$ in $\Omega$ and a vertex $v$, suppose that \[o_v\in \underset{S\in \Sqofc{\edge{T}}}\coprod \Omega(T)(S)\] is the operation associated to $(t_v,\tau)$ (here we can choose $\tau$ freely). The tree $f(v)$ is defined as the subtree of $T'$ associated to $f(o_v)$ (this definition does not depend on $\tau$).

Given a tree $T$ and an inner edge $e\in \edge{T}$ belonging to two vertices $v$ and $u$ the tree $T/e$ is defined as the tree having $\edge{T/e}=\edge{T}-\{e\}$ and  $\verx{T/e}=(\verx{T}-\{u,v\})\cup \{(u\cup v)- \{e\} \}$ and $\troot{T}$ as root.
There is a unique map in $\Omega$ from $T/e$ to $T$ which is the inclusion $\arr{\edge{T/e}}{\edge{T}}$ at the level of colours:
\[
\morp{\partial_e}{T/e}{T},
\]
which is called an \emph{inner face (map)}.

Given a vertex $v\in T$ with only one inner edge $e$ attached to it, the tree $T/v$ is defined as the one with $\edge{T/v}=\edge{T}-(v-\{e\})$ and $\verx{T/v}=\verx{T}-\{v\}$, the root of $T$ is $e$ if $\troot{T}\in v$ and $\troot{T}$ otherwise.
There is a unique map in $\Omega$ from $T/e$ to $T$ which is the inclusion $\arr{\edge{T/v}}{\edge{T}}$ at the level of colours:
\[
\morp{\partial_v}{T/v}{T},
\]
such a map is called an \emph{outer face (map)}. If $T$ is a corolla, given an edge $e\in \edge{T}$ the unique map from the empty tree to $T$ which maps the unique element of $|$ in $e$
\[
\morp{\partial_e}{|}{T},
\]
is also called an outer face.

Given a vertex $v$ with only two edges $d,e$, let us denote with $E'$ the set obtained as a quotient of $\edge{T}$ identifying $d$ and $e$ and let $\morp{\pi}{\edge{T}}{E'}$ be the canonical projection. The tree $T.v$ is defined in the following way: $E'$ is the set of edges and the vertices are the images of the vertices of $T$ along $\pi$. There is a unique map in $\Omega$ from $T$ to $T.v$ which is $\pi$ on the colours:
\[
\morp{\sigma_v}{T}{T.v},
\]
such a map is called a \emph{degeneracy (map)}.

We have then the following important characterization of morphisms in $\Omega$:
\begin{prop}[Lemma 3.1 \cite{MW07}, Theorem 2.2.6 \cite{We07}]\label{Wedeco}
Morphisms in $\Omega$ are generated by inner face maps, outer face maps, degeneracies and isomorphisms. More specifically every morphism $f$ can be factorized in $f=\sigma\theta\partial$ where $\partial$ is a composition of degeneracy map, $\theta$ is an isomorphism and $\sigma$ is a composition of face maps. 
\end{prop}

\begin{defi}
A \emph{planar structure} on a tree $T$ is a collection $\lambda=\{\lambda_v \}_{v\in \verx{T}}$ where $\lambda_v$ is a total order on $\leav{v}$. 

A \emph{planar tree} is a couple $(T,\lambda)$ where $T$ is a tree and $\lambda$ a planar structure on it.

A planar structure on $T$ induces a planar structure on every subtree $t$ of $T$ that will be denote by $\lambda_t$, and a planar structure on $\arity{T}$ that will be denoted by $\tau_\lambda$.
\end{defi}

\begin{defi}\label{defplanm}
A morphism of planar trees $\morp{f}{(T,\lambda)}{(T',\lambda')}$ is a morphism $\morp{f}{T}{T'}$ in $\Omega$ such that for every vertex $v\in \verx{T}$ the image of the operation associated to $(t_v,\lambda_v)$ is the operation associated to $(f(v),\tau_{\lambda'_{f(v)}})$.
\end{defi}

\begin{rmk}\label{ordrmk}
Recall that for every finite set $F$ of cardinality $n$ the set of total orders over $F$ is in bijective correspondence with the set of bijections from $F$ to $[n]$. 

With this identification,  the symmetric group $\Sigma_n$ acts from the left (by post-composition) on the set of total orders over $F$ (in a free and transitive way).
As a consequence the group $G_T=\underset{v\in \verx{T}}\prod \Sigma_{\card{v}}$ acts freely and transitively on the set of planar structures over $T$.
\end{rmk}

The following is easy to prove:
\begin{prop}
For every map of trees $\morp{f}{S}{T}$ and every planar structure $\lambda$ on $T$ there is a unique planar structure $\lambda'$ on $S$ that makes $f$ a morphism of planar trees, i.e. $\lambda f= \lambda'$.
\end{prop}

\begin{defi}
Consider a tree $T$ with $\card{\leav{T}}=n$, a total order $\tau$ on $\leav{T}$ and other trees $T_1,\dots,T_n$. Let $E$ be the quotient of $E'=\edge{T}\cup \edge{T_1}\cup\dots\cup \edge{T_n}$ obtained identifying $\troot{T_i}$ with $\tau^{-1}(i)$ for every $i\in [n]$ and let $\morp{\pi}{E}{E'}$ be the canonical projection. The tree $T\circ (T_1,\dots,T_n)$ (called the \emph{grafting} of $(T_1,\dots,T_n)$ over $T$) is defined as the tree having $E$ as set of edges and $\{\pi(v)\ |\ v\in \verx{T}\cup \verx{T_1}\cup\dots\cup \verx{T_n}\}$ as set of vertices. In practice $T\circ (T_1,\dots,T_n)$ should be regarded as the tree obtained by gluing the roots of $T_1\dots,T_n$ to the corresponding leaves of $T$.
\end{defi}

Let $\bar{\Omega}$ be the subcategory of $\Omega$ which has as objects all the objects of $\Omega$ and as morphisms only the ones generated by inner face maps, degeneracies and isomorphisms.

Note that in $\bar{\Omega}$ every morphism $\morp{f}{T}{T'}$ induces a morphism on the arities $\morp{f}{\arity{T}}{\arity{T'}}$ (which is always an isomorphism).

\subsubsection{$C$-trees}

\begin{defi}
A $C$-tree is a $(T,s)$ where $T$ is in $\bar{\Omega}$  and $\morp{s}{\edge{T}}{C}$ is a map of sets (called a \emph{$C$-labeling} for $T$).

A morphism of $C$-trees $\morp{f}{(A,s)}{(B,t)}$ is a morphism $\morp{f}{A}{B}$ of underlying trees such that $tf=s$ (in this equation $f$ is regarded as the underlying map on the edges).

The category of $C$-trees will be denoted $\bar{\Omega}_C$.
\end{defi}

\begin{rmk}
Given a $C$-tree $(T,s)$, the tree $\arity{T}$ and every subtree of a $C$-tree  inherit a $C$-labeling from $s$. These trees will always be considered with their induced $C$-tree structures.  

Given $C$-trees $(T,s),(T_1,s_1),\dots,(T_n,s_n)$ such that $n=\card{\leav{T}}$, and an order $\tau$ on $\leav{T}$ if $s_i(\troot{T_i})=s(\tau^{-1}(i))$ for every $i\in [n]$ the map $s\cup s_1\cup\dots\cup s_n$ factors in a $C$-labeling for $T\circ (T_1,\dots,T_n)$; under these conditions it then makes sense to talk about grafting of $C$-trees. 
\end{rmk}

\subsubsection{Push-out trees}\label{ss.pusht}

For describing the push-out (\ref{pushdiag}) we need trees that represent shapes of formal compositions of operations coming from $F(o)$, $F(x)$ and $F(y)$; for this reason we are going to put a marking on the vertices of the trees.

\begin{defi}\label{pushtdef}
A push-out $C$-tree is a couple $(T,\mf{T})$, where $T\in \bar{\Omega}_C$ and $\morp{\mf{T}}{\verx{T}}{\text{Ob}(\Gamma)}$ is a map of sets ($\Gamma$ is the category defined at the beginning of this Appendix). The map $\mf{T}$ will be called the 
\emph{marking map of $T$}.

A morphism of push-out $C$-trees $\morp{f}{(T,\mf{T})}{(T',\mf{T'})}$ is a morphism of underlying $C$-trees  $\morp{f}{T}{T'}$ such that for every $v\in \verx{T}$ and for every $u\in \verx{f(v)}$ there is a morphism in $\Gamma$ from $\mf{T}(u)$ to $\mf{T}(v)$ (N.B. this morphism goes in the "opposite direction" of $f$ since $v\in T$ and $v'\in T'$); if this morphism exists, it is unique and it will be denoted $\iota_{f,u}$.

A morphism of push-out $C$-trees will be called  \emph{marking-preserving} if for every vertex $u$ of the target $\iota_{f,u}=\id$.

A morphism of push-out $C$-trees will be called  \emph{tree-preserving} if the underlying morphism of trees is the identity.

The category of push-out trees will be denoted $\Gamma\bar{\Omega}_C$.
\end{defi}

The following proposition is straight-forward:

\begin{prop}\label{presfact}
A morphism of push-out trees $\morp{f}{(T,\mf{T})}{(S,\mf{S})}$ can always be factorized in a marking-preserving morphism followed by a tree-preserving morphism.
\end{prop}

\begin{defi}
A \emph{ordered push-out $C$-tree} is a couple $(A,\lambda_A,\tau_A)$ such that $A\in \Gamma\bar{\Omega}_C$, $\lambda_A$ is a planar structure over $A$ and $\tau_A$ is a planar structure on $\arity{A}$ (i.e. a linear order on the leaves). Note that no compatibility conditions between $\lambda_A$ and $\tau_A$ are required.

A morphism of ordered push-out $C$-trees $\morp{f}{(A,\lambda_A,\tau_A)}{(B,\lambda_B,\tau_B)}$ is just a morphism of underlying push-out $C$-trees such that the order on the leaves is respected, i.e. $\tau_b f=\tau_a$ (here $f$ stands for maps on the edges, restricted to the leaves).

The category of ordered push-out $C$-trees will be denoted $\trees{C}$.
\end{defi}

We will picture ordered push-out $C$-trees in the following way:
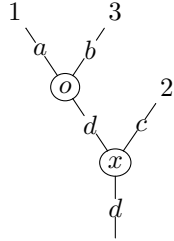
\[
 \xygraph{
!{<0cm,0cm>;<2cm,0cm>:<0cm,-4cm>::}
!{(0,0) }*+{1}="l1"
!{(0.66,0) }*+{3}="l2"
!{(0.33,0.25) }*+[o][F]{o}="2"
!{(1,0.25) }*+{2}="l3"
!{(0.66,0.5) }*+[o][F]{x}="1"
!{(0.66,0.75) }="r"
"l1"-"2"|(0.5)*+={a}
"l2"-"2"|(0.5)*+={b}
"2"-"1"|(0.5)*+={d}
"l3"-"1"|(0.5)*+={c}
"1"-"r"|(0.6)*+={d}
}
\]
where $a,b,c,d\in C$ represent the labeling on the edges, the symbols on the vertices represent the marking and the numbers on the leaves indicate the total order chosen. The edges of each vertex are ordered from the left to the right and this determines the planar structure.

\begin{rmk}\label{rmksigma}
Note that we can associate to every morphism of ordered push-out trees $\morp{f}{(A,\lambda_A,\tau_A)}{(B,\lambda_B,\tau_B)}$ a change of planar structure $\sigma_f\in G_A$ called the \emph{planar change of $f$}. It is defined as the unique one such that $\sigma_f \cdot \lambda_A=\lambda_A'$ where $\lambda_A'$ is the unique planar structure on $A$ that makes $f$ a morphism of planar trees (see Remark \ref{ordrmk}).

A morphism $f$ such that $\sigma_f=\id$ is called a \emph{planar morphism}. 

We can also associate to $f$ an element $\pi_f\in G_{\arity{A}}\simeq G_{\arity{B}}\simeq \Sigma_n$ (where $n$ is the number of leaves of $A$), called the \emph{leaves permutation of $f$}; it is defined as the planar change of the induced map $\morp{f}{(\arity{A},\tau_A,\tau_A)}{(\arity{B},\tau_B,\tau_B)}$.

Given an ordered push-out tree $(t,\lambda,\tau)$ there is a unique planar structure $\tau_\lambda$ induced on $\arity{t}$ (see Definition \ref{defplanm}). The planar change of the map $\arr{(\arity{t},\tau_\lambda,\tau)}{(\arity{t},\tau,\tau)}$ is called the \emph{twisting} of $(t,\lambda,\tau)$.
\end{rmk}

Using Proposition \ref{presfact} and \ref{Wedeco} it is easy to prove the following:

\begin{prop}\label{presfact2}
A map of ordered push-out trees $f$ can be factorized as $f=g h \theta$ where $h$ is planar tree-preserving, $g$ is a composition of planar marking-preserving inner-face and degeneracy morphisms and $\theta$ is a marking-preserving isomorphism.
\end{prop}

\begin{rmk}
Every signature in $C$ can be seen as a $C$-corolla. 
This identification associates to $S=(c_1,\dots,c_n;c)\in \Sqofc{C}$ the triple $(C_n,s_S,\tau)$ where $C_n$  is the $n$-corolla, $\morp{s_S}{\leav{C_n}}{C}$ is the $C$-labeling such that $s_S(i)=c_i$ for every $n\in [n]$ and $s_S(\ast)=c$ and $\tau$ is the usual order on $[n]=\leav{C_n}$. This corolla will be also denoted by $S$. With $S_x$,$S_y$,$S_o$ we will indicate $S$ together with a marking $x,y$ or $o$ on the unique vertex. 

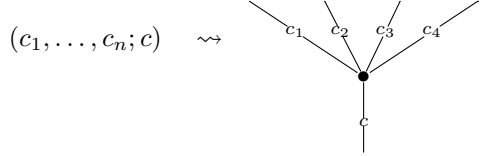
\[
\begin{aligned}
\\
\\
 (c_1,\dots,c_n;c) \quad \rightsquigarrow
\\
\end{aligned} 
\quad
\xygraph{
!{<0cm,0cm>;<3cm,0cm>:<0cm,-2.0cm>::}
!~-{@{-}@[|(8.5)]}
!{(0,0) }="1"
!{(0.33,0)}="2"
!{(0.66,0)}="3"
!{(1,0)}="4"
!{(0.5,0.5)}*{\bullet}="v"
!{(0.5,1)}="r"
"1"-"v"|(0.4){c_1}
"2"-"v"|(0.4){c_2}
"3"-"v"|(0.4){c_3}
"4"-"v"|(0.4){c_4}
"r"-"v"|(0.4){c}
}
\]

Conversely to every ordered $C$-corolla $(A,\tau)$ we can associate a signature in $C$.

For every ordered $C$-tree $(T,\lambda,\tau)$ the signature corresponding to $(\arity{T},\tau)$ will be also denoted by $\arity{T}$. 
\end{rmk}

\begin{rmk}\label{treecomp}
Let us forget for a moment about the marking on the vertices. A $C$-tree $t$ together with a planar structure $\lambda_t$ and an order on the leaves $\tau_t$ should be regarded
as a scheme of composition for operations in $C$-coloured operads. This means that given a $C$-coloured operad $\mathcal{O}$ (in $\Set$) and an operation in
$\mathcal{O}$ of signature $a(\arity{v})$ (i.e. an element of $\mathcal{O}(\arity{v})$) for every $v\in \verx{t}$, we should be able to compose them following the shape of $t$
and get an operation in $\mathcal{O}$ of signature $\arity{t}$. 
More formally given a $C$-coloured operad $\mathcal{O}$ we can build a map
\[
\morp{\Gamma_t}{\underset{v\in \verx{t}}\bigotimes \mathcal{O}(\arity{v})}{\mathcal{O}(\arity{t})}.
\]

that will be called \emph{composition along $t$}.

It is defined by induction on $m$, the number of vertices of $t$:
\begin{itemize}
\item If $m=0$ then $t=|_c$, then $\mathcal{O}(\arity{t})=I$ and the map is just the identity of $c$ in $\mathcal{O}$:
\[
\morp{u_c}{I}{\mathcal{O}(c;c)};
\]
\item If $m=1$ then $t$ is a corolla and the map is the symmetry map 
\[
\morp{\sigma^*}{\mathcal{O}(\arity{v})}{\mathcal{O}(\arity{t})},
\]
where $\sigma\in \Sigma_{\card{v}}$ is the unique one such that $\sigma \lambda_{t,v}=\tau_t$;
\item If $m>1$ suppose that $\Gamma_t$ was already defined for all the trees such that $\card{\verx{t}}<m$. Suppose that the decomposition of $t$ is $t=s\circ (s_1,\dots,s_n)$ and that $\sigma$ is the twisting of $(t,\lambda_t,\tau_t)$, then

\[\underset{v\in verx{t}}\bigotimes \mathcal{O}(\arity{v})\simeq 
\underset{i\in [n]}\bigotimes (\underset{v\in \verx{s_i}}\bigotimes \mathcal{O}(\arity{v}))\otimes \bigotimes (\underset{u\in \verx{s}}\bigotimes \mathcal{O}(\arity{u}))\]

(the isomorphism is given by symmetric isomorphisms of $\Mm$); the morphism $\Gamma_t$ is defined as $\sigma^{-1*}\gamma \circ(\underset{i\in [n]}\bigotimes \Gamma_{s_i})\otimes \Gamma_s$, where 
\[
\morp{\gamma}{(\underset{i\in [n]}\bigotimes \mathcal{O}(\arity{s_i}))\otimes \mathcal{O}(\arity{s})}{\mathcal{O}(\arity{t,\tau\lambda_t,\tau\lambda_t})}
\]
is the usual composition in $\mathcal{O}$. 
\end{itemize}

It is easy to check that for every morphism of $C$-coloured operads $\morp{f}{\mathcal{O}}{\mathcal{P}}$ the equality $f_{\arity{t}}\Gamma_t=\Gamma_t \circ (\underset{v\in \verx{t}}\bigotimes f_{\arity{v}})$ holds. 

\end{rmk}

\begin{defi} For every $S\in \Sqofc{C}$, by $\trees{C}_S$ we will denote the subcategory of $\trees{C}$ spanned by all objects $(A,\lambda_A,\tau_A)\in \trees{C}$ such that $(\arity{A},\tau_A)=S$.
\end{defi}

\subsection{Description of push-outs}\label{secpushfc}

Recall that we want to describe push-outs in the category $\EOperfc{\Mm}{C}$, i.e. colimits of functors $\morp{F}{\Gamma}{\EOperfc{\Mm}{C}}$, as explained at the beginning of Section \ref{sec1}. We will denote $F(o)$,$F(x)$ and $F(y)$ by $F_o$,$F_x$ and $F_y$ respectively.

\subsubsection{Digression on $\LaxSCat{\Mm}$}

Let $\LaxSCat{\Mm}$ be the category defined as follows: the objects are the functors of the  kind $\morp{F}{\mathcal{C}}{\Mm}$ where $\mathcal{C}$ is a ( not fixed) small category and a morphism  between $\morp{F}{\mathcal{C}}{\Mm}$ and $\morp{G}{\mathcal{D}}{\Mm}$ is a couple $(h,\alpha)$ where $\morp{h}{\mathcal{C}}{\mathcal{D}}$ and $\alpha$ is a natural transformation between $F$ and $G\circ h$. Compositions and identities are defined in the evident way.


Since $\Mm$ is a symmetric monoidal category, $\LaxSCat{\Mm}$ comes endowed with a tensor product
$\morp{F\otimes G}{\mathcal{C}\times \mathcal{D}}{\Mm}$ such that $F\otimes G((C,D))=F(C)\otimes F(D)$.
The unit of $\otimes$ is the functor $\morp{\hat{I}}{\ast}{\Mm}$ whose image is $I$, the unit of $\Mm$.
 
There is a functor 

\[
\gfun{\varinjlim}{\Cat//\Mm}{\Mm}{F}{\varinjlim F},
\]
which has a right adjoint $\tilde{(-)}$ which associates to every $V\in \Mm$ the functor 
\[
 \gfun{\tilde{V}}{\ast}{\Mm}{\ast}{V\, .}
\]

The functors $\varinjlim$ and $\tilde{(-)}$ are both strong monoidal since $\Mm$ is monoidal closed.

From a ($C$-coloured) operad $\mathcal{O}=\{\mathcal{O(S)}\}_{S\in \Sqofc{C}}$ in $\LaxSCat{\Mm}$ we get a ($C$-coloured) operad $\varinjlim\mathcal{O}=\{\varinjlim \mathcal{O(S)}\}_{S\in \Sqofc{C}}$ in $\Mm$.

\subsubsection{Description of $Z$} We are now going to describe an operad $\hat{Z}$ in $\LaxSCat{\Mm}$ the colimit of which will be the desired push-out $Z$ (see diagram \ref{pushdiag}).

For every $S\in \Sqofc{C}$ the $S$-component of $\hat{Z}$ is given by the functor:
\begin{equation}\label{tree.functor}
\gfun{\hat{Z}(S)}{\opc{\trees{C}_S}}{\Mm}{(T,\lambda_T,\tau_T)}{\underset{v\in \verx{T}}\bigotimes F_{\mf{T}(v)}(\arity{v})}
\end{equation}
Some pictures at this moment might be helpful. For example the functor $\hat{Z}$ associate to the tree on the left the tensor product on the right:
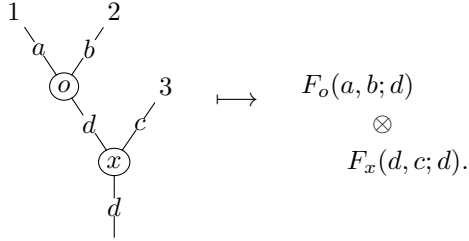
\[
\xygraph{
!{<0cm,0cm>;<2cm,0cm>:<0cm,-4cm>::}
!{(0,0) }*+{1}="l1"
!{(0.66,0) }*+{2}="l2"
!{(0.33,0.25) }*+[o][F]{o}="2"
!{(1,0.25) }*+{3}="l3"
!{(0.66,0.5) }*+[o][F]{x}="1"
!{(0.66,0.75) }="r"
"l1"-"2"|(0.5)*+={a}
"l2"-"2"|(0.5)*+={b}
"2"-"1"|(0.5)*+={d}
"l3"-"1"|(0.5)*+={c}
"1"-"r"|(0.6)*+={d}
}
\quad
\xygraph{
!{<0cm,0cm>;<2cm,0cm>:<0cm,-4cm>::}
!{(0.5,0.3) }*+{\longmapsto}="1"
}
\quad
\xygraph{
!{<0cm,0cm>;<2cm,0cm>:<0cm,-4cm>::}
!{(0,0) }="l1"
!{(0.66,0) }="l2"
!{(0.33,0.25) }*+{F_o(a,b;d)}="2"
!{(1,0.25) }="l3"
!{(0.66,0.5) }*+{F_x(d,c;d).}="1"
!{(0.66,0.75) }="r"
"2"-@{}"1" |(0.5)*+={\otimes}
}
\]


For the definition of $\hat{Z}(S)$ on a morphism $\morp{f}{(A,\lambda_A,\tau_A)}{(B,\lambda_B,\tau_B)}$ let us consider several cases:

\begin{itemize}
\item if $f$ is planar and tree-preserving then $\hat{Z}(S)(f)$ is
$\otimes\, F(\iota_{f,u})_{\arity{u}}$ ( recall that the $\iota_{f,u}$'s were defined in Definition \ref{pushtdef}) ;
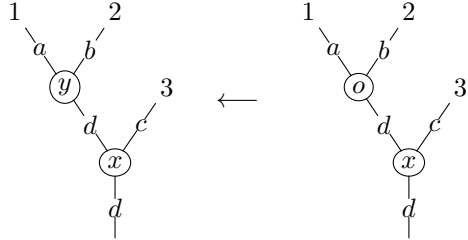
\[
\xygraph{
!{<0cm,0cm>;<2cm,0cm>:<0cm,-4cm>::}
!{(0,0) }*+{1}="l1"
!{(0.66,0) }*+{2}="l2"
!{(0.33,0.25) }*+[o][F]{y}="2"
!{(1,0.25) }*+{3}="l3"
!{(0.66,0.5) }*+[o][F]{x}="1"
!{(0.66,0.75) }="r"
"l1"-"2"|(0.5)*+={a}
"l2"-"2"|(0.5)*+={b}
"2"-"1"|(0.5)*+={d}
"l3"-"1"|(0.5)*+={c}
"1"-"r"|(0.6)*+={d}
}
\quad
\xygraph{
!{<0cm,0cm>;<2cm,0cm>:<0cm,-4cm>::}
!{(0.5,0.3) }*+{\longleftarrow}="1"
}
\quad
\xygraph{
!{<0cm,0cm>;<2cm,0cm>:<0cm,-4cm>::}
!{(0,0) }*+{1}="l1"
!{(0.66,0) }*+{2}="l2"
!{(0.33,0.25) }*+[o][F]{o}="2"
!{(1,0.25) }*+{3}="l3"
!{(0.66,0.5) }*+[o][F]{x}="1"
!{(0.66,0.75) }="r"
"l1"-"2"|(0.5)*+={a}
"l2"-"2"|(0.5)*+={b}
"2"-"1"|(0.5)*+={d}
"l3"-"1"|(0.5)*+={c}
"1"-"r"|(0.6)*+={d}
}
\quad
\]
\[
\xygraph{
!{<0cm,0cm>;<2cm,0cm>:<0cm,-4cm>::}
!{(0,0) }="l1"
!{(0.66,0) }="l2"
!{(0.33,0.25) }*+{F_o(a,b;d)}="2"
!{(1,0.25) }="l3"
!{(0.66,0.5) }*+{F_x(d,c;d)}="1"
!{(0.66,0.75) }="r"
"2"-@{}"1" |(0.5)*+={\otimes}
}
\quad
\xygraph{
!{<0cm,0cm>;<2cm,0cm>:<0cm,-4cm>::}
!{(0.5,0.4) }*+{\longrightarrow}="1"
}
\quad
\xygraph{
!{<0cm,0cm>;<2cm,0cm>:<0cm,-4cm>::}
!{(0,0) }="l1"
!{(0.66,0) }="l2"
!{(0.33,0.25) }*+{F_y(a,b;d)}="2"
!{(1,0.25) }="l3"
!{(0.66,0.5) }*+{F_x(d,c;d)}="1"
!{(0.66,0.75) }="r"
"2"-@{}"1" |(0.5)*+={\otimes}
}
\]
\item If $f$ is marking-preserving and the underlying morphism $\morp{f}{A}{B}$ is an isomorphism, $\hat{Z}(S)(f)$ is defined as $\theta (\underset{v\in \verx{A}}\otimes{\sigma_{f,v}^*})$ where $\sigma_{f,v}^*$ is the symmetry morphism on the $\arity{v}$ component of the operad $F_{\mf{v}}$ associated to $\sigma_{f,v}\in \Sigma_{\card{v}}$ (see Remark \ref{rmksigma}); $\theta$ is just the appropriate symmetric isomorphism of $\Mm$ (that comes with the structure of symmetric monoidal category);
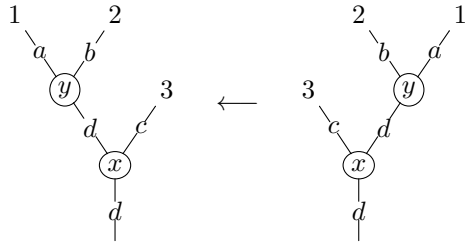
\[
\xygraph{
!{<0cm,0cm>;<2cm,0cm>:<0cm,-4cm>::}
!{(0,0) }*+{1}="l1"
!{(0.66,0) }*+{2}="l2"
!{(0.33,0.25) }*+[o][F]{y}="2"
!{(1,0.25) }*+{3}="l3"
!{(0.66,0.5) }*+[o][F]{x}="1"
!{(0.66,0.75) }="r"
"l1"-"2"|(0.5)*+={a}
"l2"-"2"|(0.5)*+={b}
"2"-"1"|(0.5)*+={d}
"l3"-"1"|(0.5)*+={c}
"1"-"r"|(0.6)*+={d}
}
\quad
\xygraph{
!{<0cm,0cm>;<2cm,0cm>:<0cm,-4cm>::}
!{(0.5,0.3) }*+{\longleftarrow}="1"
}
\quad
\xygraph{
!{<0cm,0cm>;<2cm,0cm>:<0cm,-4cm>::}
!{(1,0) }*+{1}="l1"
!{(0.33,0) }*+{2}="l2"
!{(0.66,0.25) }*+[o][F]{y}="2"
!{(0,0.25) }*+{3}="l3"
!{(0.33,0.5) }*+[o][F]{x}="1"
!{(0.33,0.75) }="r"
"l1"-"2"|(0.5)*+={a}
"l2"-"2"|(0.5)*+={b}
"2"-"1"|(0.5)*+={d}
"l3"-"1"|(0.5)*+={c}
"1"-"r"|(0.6)*+={d}
}
\quad
\]
\[
\xygraph{
!{<0cm,0cm>;<2cm,0cm>:<0cm,-4cm>::}
!{(0,0) }="l1"
!{(0.66,0) }="l2"
!{(0.33,0.25) }*+{F_y(a,b;d)}="2"
!{(1,0.25) }="l3"
!{(0.66,0.5) }*+{F_x(d,c;d)}="1"
!{(0.66,0.75) }="r"
"2"-@{}"1" |(0.5)*+={\otimes}
}
\quad
\xygraph{
!{<0cm,0cm>;<2cm,0cm>:<0cm,-4cm>::}
!{(0.5,0.4) }*+{\longrightarrow}="1"
}
\quad
\xygraph{
!{<0cm,0cm>;<2cm,0cm>:<0cm,-4cm>::}
!{(0.66,0.25) }*+{F_y(b,a;d)}="2"
!{(0,0.25) }="l3"
!{(0.33,0.5) }*+{F_x(c,d;d)}="1"
!{(0.33,0.75) }="r"
"2"-@{}"1" |(0.5)*+={\otimes}
}
\]

\item If $f$ is planar ,marking-preserving and the underlying morphism of trees $\morp{f}{A}{B}$ is an inner face $\partial_e$, then $\hat{Z}(S)(f)$ is defined using the composition of the operad which marks the extremes of $e$;
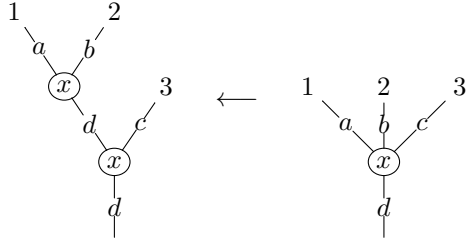
\[
\xygraph{
!{<0cm,0cm>;<2cm,0cm>:<0cm,-4cm>::}
!{(0,0) }*+{1}="l1"
!{(0.66,0) }*+{2}="l2"
!{(0.33,0.25) }*+[o][F]{x}="2"
!{(1,0.25) }*+{3}="l3"
!{(0.66,0.5) }*+[o][F]{x}="1"
!{(0.66,0.75) }="r"
"l1"-"2"|(0.5)*+={a}
"l2"-"2"|(0.5)*+={b}
"2"-"1"|(0.5)*+={d}
"l3"-"1"|(0.5)*+={c}
"1"-"r"|(0.6)*+={d}
}
\quad
\xygraph{
!{<0cm,0cm>;<2cm,0cm>:<0cm,-4cm>::}
!{(0.5,0.3) }*+{\longleftarrow}="1"
}
\quad
\xygraph{
!{<0cm,0cm>;<2cm,0cm>:<0cm,-4cm>::}
!{(0,0.25) }*+{1}="l1"
!{(0.5,0.25) }*+{2}="l2"
!{(1,0.25) }*+{3}="l3"
!{(0.5,0.5) }*+[o][F]{x}="1"
!{(0.5,0.75) }="r"
"l1"-"1"|(0.5)*+={a}
"l2"-"1"|(0.5)*+={b}
"l3"-"1"|(0.5)*+={c}
"1"-"r"|(0.6)*+={d}
}
\]
\[
\xygraph{
!{<0cm,0cm>;<2cm,0cm>:<0cm,-4cm>::}
!{(0,0) }="l1"
!{(0.66,0) }="l2"
!{(0.33,0.25) }*+{F_x(a,b;d)}="2"
!{(1,0.25) }="l3"
!{(0.66,0.5) }*+{F_x(d,c;d)}="1"
!{(0.66,0.75) }="r"
"2"-@{}"1" |(0.5)*+={\otimes}
}
\quad
\xygraph{
!{<0cm,0cm>;<2cm,0cm>:<0cm,-4cm>::}
!{(0.5,0.4) }*+{\longrightarrow}="1"
}
\quad
\xygraph{
!{<0cm,0cm>;<2cm,0cm>:<0cm,-4cm>::}
!{(0.5,0.4) }*+{F_x(a,b,c;d)}="1"
}
\quad
\]
\item If $f$ is planar ,marking-preserving and the underlying morphism of trees $\morp{f}{A}{B}$ is a degeneracy map $\sigma_v$ then $\hat{Z}(S)$ is defined using the natural unit isomorphism of $\Mm$ and the identity of $F_{\mf{v}}(\arity{v})$ (note that $\arity{v}=(c;c)$ for some $c\in C$).
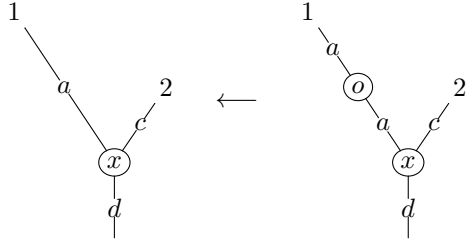
\[
\xygraph{
!{<0cm,0cm>;<2cm,0cm>:<0cm,-4cm>::}
!{(0,0) }*+{1}="l1"
!{(1,0.25) }*+{2}="l3"
!{(0.66,0.5) }*+[o][F]{x}="1"
!{(0.66,0.75) }="r"
"l1"-"1"|(0.5)*+={a}
"l3"-"1"|(0.5)*+={c}
"1"-"r"|(0.6)*+={d}
}
\quad
\xygraph{
!{<0cm,0cm>;<2cm,0cm>:<0cm,-4cm>::}
!{(0.5,0.3) }*+{\longleftarrow}="1"
}
\quad
\xygraph{
!{<0cm,0cm>;<2cm,0cm>:<0cm,-4cm>::}
!{(0,0) }*+{1}="l1"
!{(0.33,0.25) }*+[o][F]{o}="2"
!{(1,0.25) }*+{2}="l3"
!{(0.66,0.5) }*+[o][F]{x}="1"
!{(0.66,0.75) }="r"
"l1"-"2"|(0.5)*+={a}
"2"-"1"|(0.5)*+={a}
"l3"-"1"|(0.5)*+={c}
"1"-"r"|(0.6)*+={d}
}
\]
\[
\xygraph{
!{<0cm,0cm>;<2cm,0cm>:<0cm,-2.5cm>::}
!{(0,0) }*+{}="l1"
!{(0.66,0) }*+{}="l2"
!{(0.33,0.25) }*+{}="2"
!{(1,0.25) }*+{}="l3"
!{(0.66,0.5) }*+{F_x(a,c;d)}="1"
!{(0.66,0.75) }="r"
}
\quad
\xygraph{
!{<0cm,0cm>;<2cm,0cm>:<0cm,-2.5cm>::}
!{(0.5,0.4) }*+{\longrightarrow}="1"
}
\quad
\xygraph{
!{<0cm,0cm>;<2cm,0cm>:<0cm,-2.5cm>::}
!{(0.33,0.25) }*+{I}="2"
!{(0.66,0.5) }*+{F_x(a,c;d)}="1"
!{(0.66,0.75) }="r"
"2"-@{}"1"|(0.5)*+={\otimes}
}
\quad
\xygraph{
!{<0cm,0cm>;<2cm,0cm>:<0cm,-2.5cm>::}
!{(0.5,0.4) }*+{\longrightarrow}="1"
}
\quad
\xygraph{
!{<0cm,0cm>;<2cm,0cm>:<0cm,-2.5cm>::}
!{(0.33,0.25) }*+{F_o(a;a)}="2"
!{(0.66,0.5) }*+{F_x(a,c;d)}="1"
!{(0.66,0.75) }="r"
"2"-@{}"1"|(0.5)*+={\otimes}
}
\]
\end{itemize}

These four cases fully determined the definition of $\hat{Z}(S)$ on morphisms (Proposition \ref{presfact2}). It is routine to check that this is a good definition.

We want to define a $C$-operad structure on the collection $\{\hat{Z}(S)\}_{S\in \Sqofc{C}}$.

First notice that for every signature $S\in\Sqofc{C}$ and every $\sigma\in\Sigma_{\card{S}}$ there is a morphism $\morp{\sigma^*}{\hat{Z}(S)}{\hat{Z}(\sigma^{-1} (S))}$ given by the couple $(F_\sigma,\mu_\sigma)$. The functor $\morp{F_\sigma}{\opc{\trees{C}_S}}{\opc{\trees{C}_{\sigma^{-1}S}}}$ sends a tree $(T,\lambda_T,\tau_T)$ to $(T,\lambda_T,\sigma^{-1}\tau_T)$ and sends a morphism to the unique one having the same underlying morphism of (push-out $C$-labeled) trees. The natural transformation $\ntra{\mu_\sigma}{\hat{Z}(S)}{\hat{Z}(\sigma^{-1}S) F_\sigma}$ is just the identity on every object.

\begin{prop}
There is an operad structure on the collection $\{\hat{Z}(S)\}_{S\in \Sqofc{C}}$. 
\end{prop}
\begin{proof}
For every $n\in \N$, for every $S\in \Sqofc{C}$ such that $\card{S}=n$, and for every $S_1,\dots,S_n \in \Sqofc{C}$ such that $(\troot{S_1},\dots,\troot{S_n})=\leav{S}$ we can define the composition morphism:
\[
\morp{\circ}{\hat{Z}(S_1)\otimes \dots \otimes\hat{Z}(S_n)\otimes \hat{Z}(S)}{\hat{Z}(S\circ(S_1,\dots,S_n))}
\]

as the couple $(\gamma,\mu)$, where $\gamma$ is the functor
\[
\gfun{\gamma}{\opc{\trees{C}_{S_1}}\times \dots \times \opc{\trees{C}_{S_n}}\times \opc{\trees{C}_{S}}}{\opc{\trees{C}_{S\circ (S_1,\dots,S_n)}}}{(T_1,\dots,T_n)}{T\circ (T_1,\dots,T_n)}
\]
(we leave to the reader the definition of $\gamma$ on the morphisms).
The natural transformation $\ntra{\mu}{\hat{Z}(S_1)\otimes \dots \otimes\hat{Z}(S_n)\otimes \hat{Z}(S)}{\hat{Z}(S\circ(S_1,\dots,S_n))}$ is defined using  symmetric isomorphisms of $\Mm$.

The identity on the colour $c\in C$ is given by the morphism:
\[
\morp{\id_C=(G,\iota)}{\tilde{I}}{\hat{Z}(c;c)},
\]
where $\morp{G}{\ast}{\opc{\trees{C}_{(c;c)}}}$ has as image $|_c$, the empty tree labeled by $c$. The natural transformation $\iota$ is just an arrow $\arr{I=\tilde{I}(\ast)}{\hat{Z}(c;c)(|_c)=I}$, which is defined to be the identity of $I$.

For every signature $S\in\Sqofc{C}$ the right action of $\Sigma_{\card{S}}$ is given for every $\sigma\in \Sigma_{\card{S}}$ by the morphism:
\[
\morp{\sigma^*}{\hat{Z}(S)}{\hat{Z}(\sigma^{-1} (S))},
\]
described above.
We leave to the reader to check that this defines an operad structure, i.e. to check associativity, equivariance and unitality.
\end{proof}

\begin{cor}
The collection $Z=\{\varinjlim{\hat{Z}}(S)\}_{S\in \Sqofc{C}}$ has an operad structure.
\end{cor}

\begin{rmk}\label{treecompz}
From Remark \ref{treecomp}, for every planar tree with an order on the leaves we have "composition" morphisms:
\[
\morp{(H_t,\nu_t)=:\Gamma^{\hat{Z}}_t}{\underset{v\in \verx{t}}\bigotimes \hat{Z}(\arity{v})}{\hat{Z}(\arity{t})}
\]
and 
\[
\morp{\varinjlim{\Gamma^{\hat{Z}}_t}=\Gamma^{Z}_t}{\underset{v\in \verx{t}}\bigotimes Z(\arity{v})}{Z(\arity{t})}.
\]

It is routine to check that the following diagram commutes:

\begin{equation}\label{diagcomplim}
\diagc{\underset{v\in \verx{t}}\bigotimes \hat{Z}(\arity{v})(t_v)}{\hat{Z}(\arity{t})(t)}{\underset{v\in \verx{t}}\bigotimes Z(\arity{v})}{Z(\arity{t}).}{}{\nu_t}{}{\Gamma^{Z}_t}{}
\end{equation}
\end{rmk}

There are morphisms of collections from $\morp{p'}{\tilde{F_x}}{\hat{Z}}$ and $\morp{q'}{\tilde{F_y}}{\hat{Z}}$ (recall that $\tilde{(-)}$ is the right adjoint of $\varinjlim$). On the $S$-operation $p$ is defined as $p'_s=(P_S,\mu_S)$ where $P(\ast)=S_x$ and $\mu_S$ is just the identity. In a similar way $q'_S=(Q_S,\nu_S)$ with $Q(\ast)=S_y$ and $\nu_S$ which is just the identity.

These maps are not maps of operads (at least not in a strict sense), but they become so in $\Mm$.

In fact, applying $\varinjlim$ to $p'$ and $q'$ we obtain maps $\morp{p}{F_x}{Z}$ and $\morp{q}{F_y}{Z}$. These maps make diagram \ref{pushdiag} commute and are in fact maps of operads.

\begin{prop}
The morphisms $\morp{p}{F_x}{\hat{Z}}$ and $\morp{q}{F_y}{\hat{Z}}$ are morphisms of operads.
\end{prop}
\begin{proof}
Consider $s,s_1,\dots,s_n\in \Sqofc{C}$ such that $s\circ(s_1,\dots,s_n)$ makes sense and let $t=s\circ(s_1,\dots,s_n)$ (where the grafting is view as a grafting of trees) then the following diagram commutes:

\[
\xymatrix{(\underset{i\in [n]} \bigotimes F_x(s_i))\otimes F_x(s)\ar[r]^{p'} \ar[d]^{\Gamma_t^{F_x}} & 
(\underset{i\in [n]} \bigotimes \hat{Z}(s_i)(s_{i,x}))\otimes \hat{Z}(s)(s_{x}) \ar[r]^{} \ar[d]^{\nu_t}& (\underset{i\in [n]} \bigotimes Z(s_i))\otimes Z(s) \ar[d]^{\Gamma^{\hat{Z}}{t}}\\
F_x(s\circ(s_1,\dots,s_n))\ar[r]^-{p'} & \hat{Z}(\arity{t})(\arity{t}_x) \ar[r]^{} & Z(\arity{t})\ .
}
\]
In fact the square on the right is just (\ref{diagcomplim}) and the outer square shows that $p$ respects the composition.

For every signature $S\in \Sqofc{C}$ and every $\sigma\in \Sigma_{\card{S}}$ if in the diagram above one takes as tree the corolla associated to $S$ with a twisting $\sigma$ (see Remark \ref{rmksigma}) instead of $t$ one obtains the compatibility of $p$ with the action of $\sigma$.

In the same way taking $|_c$, the empty tree labeled with $c\in C$, instead of $t$ one obtains the preservation of the identity of $c$.

This proves that $p$ is a map of operads. The case of $q$ is analogous.

\end{proof}

\begin{prop}\label{exp.push}
Diagram (\ref{pushdiag}) with $Z$, $p$ and $q$ defined as above, is a push-out diagram in the category $\EOperfc{\Mm}{C}$.
\end{prop}
\begin{proof}
Suppose another $C$-coloured operad $L$ is given, with morphism $\morp{h}{F_x}{L}$ and $\morp{k}{F_y}{L}$ such that $hF(\iota_2)=kF(\iota_1)$.
We have to define a morphism of $C$-coloured operads  $\morp{l}{Z}{L}$ such that $lp=h$ and $lq=k$.

For every $S\in \Sqofc{C}$, to define a morphism $\morp{l_S}{Z(S)}{L(S)}$ is equivalent to give a morphism from $\morp{l_{S,t}}{\hat{Z}(S)(t)}{L(S)}$ for every $t\in \trees{C}_S$
in a compatible way (for every morphism $\morp{i}{t}{t'}$ the relation $l_{S,t}=l_{S,t'} Z(S)(i)$ has to hold).

Recall that $\hat{Z}(S)(t)=\underset{v\in \verx{t}}\bigotimes \mf{t}(v)(\arity{v})$; the morphism $l_{S,t}$ is then defined as the composition of
\[
\morp{\underset{v\in \verx{t}}\bigotimes w_v}{Z(S)(t)}{\underset{v\in \verx{t}}\bigotimes L(\arity{v})}
\]
 and 
 \[
 \morp{\Gamma_t}{\underset{v\in \verx{t}}\bigotimes L(\arity{v})}{L(S)}, 
\]
where:
\[
w_v=
\begin{cases}
 h_{\arity{v}} & \text{if } \mf{t}(v)=x\\
 k_{\arity{v}} & \text{if } \mf{t}(v)=y\\
 yk_{\arity{v}} &\text{if } \mf{t}(v)=o
\end{cases}
\]
and $\Gamma_t$ is the composition of $L$ obtained following the shape of $t$ (seen as a planar tree without marking on the vertices) as defined in Remark \ref{treecomp}.

It can be checked that this is a good definition and the obtained map is a morphism of operads.

It remains to check that $l$ is unique. Suppose that there is another $l'$ such that $l'p=h$ and $l'q=k$;
this implies that for every  $S\in \Sqofc{C}$ one has to define $l'_{S,S_x}=h_S=l_{S,S_x}$ and $l'_{S,S_y}=k_S=l_{S,S_y}$.

For every $S\in \Sqofc{C}$ and $t\in Z(S)$ the following diagram has to commute:
\[
 \xymatrix{
           \underset{v\in \verx{t}}\bigotimes \hat{Z}(\arity{v})(t_v) \ar[dd]_{\underset{v}\otimes l'_{\arity{v},t_v}}\ar[r]^-{\simeq}_-{\nu_t} & \hat{Z}(S)(t) \ar[dd]^{l'_{S,t}} \\
            & \\
           \underset{v\in \verx{t}}\bigotimes L(\arity{v}) \ar[r]^{\Gamma_t} & L(S),
          }
\]
where $\nu_t$ is defined as in Remark \ref{treecompz}.

As we see, the vertical arrow on the left has to be equal to $\underset{v}\otimes l_{\arity{v},\arity{v}_{M_t(v)}}$ so $l'_{S,t}$ is forced to 
be defined as $l_{S,t}$.
 
\end{proof}

\subsection{Push-outs along fully-faithful inclusions}

Let now set ourself in the category of coloured operads, with no set of colours fixed. We want to prove the following fact:
\begin{prop}\label{prop.ap.ff}
Suppose we have a push-out diagram in $\EOper{\Mm}$ ($\ENSOper{\Mm}$ or $\EROper{\Mm}$):
\[
 \diagc{A}{P}{B}{Q.}{i}{f}{h}{g}{}
\]
If $i$ and $f$ are injective on colours and $i$ is fully-faithful then $h$ is injective on colours and is fully-faithful.
\end{prop}

\begin{proof}
As we remarked at the beginning of the Appendix we will give a proof only for $\EOper{\Mm}$.

Recall that since the colour functor from $\EOper{\Mm}$ to $\Set$ preserve colimits, at the level of colours the diagram
\[
 \diagc{\Cl{A}}{\Cl{P}}{\Cl{B}}{\Cl{Q}}{i}{f}{h}{g}{}
\]
is a push-out diagram. The first part of the statement then follows from the fact that in $\Set$ the push-out of an injective function is injective.
Under our hypotheses all the maps in the diagram are injective so we can suppose that $\Cl{B}$,$\Cl{P}$,$\Cl{A}$ are subset of  $C=\Cl{Q}$ such that
$|\Cl{B}\cup \Cl{P}|=\Cl{Q}$ and $\Cl{B}\cap \Cl{P}=\Cl{A}$.

The only thing we have to prove is that $h$ is fully-faithful, i.e. $\morp{\trf{h}}{P}{h^*(Q)}$ (notation of Section 3) is an isomorphism in $\EOperfc{\Mm}{\Cl{P}}$.

Recall from the previous discussion about colimits in the category of $\EOper{\Mm}$ that our diagram is a push-out if and only if the following one 
\begin{equation}\label{fullypush2}
 \diagc{g_! i_! A}{h_! P}{g_! B}{Q}{g_!(\trc{i})}{h_!(\trc{f})}{\trc{h}}{\trc{g}}{}
\end{equation}
is a push-out diagram in $\EOperfc{\Mm}{\Cl{Q}}$.

In $\EOperfc{\Mm}{\Cl{P}}$ we have the equality $\trf{h}=\eta_P h^*(\trc{h})$, where $\eta$ is the unit of the adjunction $(h_!,h^*)$ . Since $h$ is injective on colours the unit is a natural isomorphism, 
so we are reduced to prove that $h^*(\trc{h})$ is an isomorphism.

Since all the underlying maps of colours we are considering are injective we have the following descriptions of the operads in diagram \ref{fullypush2}:

For every $S=(s_1,\dots,s_n;s)\in \Sqofc{C}$:
\[
 (g_! i_! A)(S)=
\begin{cases}
 A(S) & \text{if } s,s_i\in \Cl{A}\ \forall i\in [n]\\
 I & \text{if }  n=1 \text{ and } s_1=s\in C/\Cl{A}\\
 \varnothing & \text{otherwise .}
\end{cases}
\]
\[
 (g_! B)(S)=
\begin{cases}
 B(S) & \text{if } s,s_i\in \Cl{B}\ \forall i\in [n]\\
 I & \text{if }  n=1 \text{ and } s_1=s\in C-\Cl{B}\\
 \varnothing & \text{otherwise}
\end{cases}
\]
\[
 (h_! P)(S)=
\begin{cases}
 P(S) & \text{if } s,s_i\in \Cl{P}\ \forall i\in [n]\\
 I & \text{if }  n=1 \text{ and } s_1=s\in C-\Cl{P}\\
 \varnothing & \text{otherwise}
\end{cases}
\]

To prove that $h^*(\trc{h})$ is an isomorphism is exactly to prove that for every $S=(s_1,\dots,s_n;s)\in \Sqofc{\Cl{P}}$ the morphism
$\morp{h_S}{h_!P(S)=P(S)}{Q(S)}$ is an isomorphism.

Fix $S=(s_1,\dots,s_n;s)\in \Sqofc{\Cl{P}}$. From the description of push-outs given in section \ref{secpushfc} we get that $Q(S)=\varinjlim \hat{Z}(S)$ for a functor
$\morp{\hat{Z}(S)}{\trees{C}_S^{op}}{\Mm}$, where $\trees{C}_S$ is the category of ordered push-out trees with arity $S$ (in which this time we can suppose that the vertices
are marked in $\{A,B,P\}$).

The map $\morp{h_S}{h_!P(S)=P(S)}{Q(S)}$ is exactly the canonical map from $\arr{\hat{Z}(S)(S_P)}{Q(S)}$.

We are now going to do a series of changes of the index category of the colimit $\varinjlim \hat{Z}(S)$ in order to show that it is canonically isomorphic to $P(S)$.

First, consider $\trees{C}_S'$ the full subcategory of $\trees{C}_S$ spanned by all the trees $t$ such that for every $v\in \verx{t}$:
\begin{itemize}
 \item if $\mf{t}(v)=A$ then $\arity{v}\in \Sqofc{\Cl{A}}$ or $\arity{v}=(c;c)$ for some $c\in C$;
 \item if $\mf{t}(v)=B$ then $\arity{v}\in \Sqofc{\Cl{B}}$ or $\arity{v}=(c;c)$ for some $c\in C$;
 \item if $\mf{t}(v)=P$ then $\arity{v}\in \Sqofc{\Cl{P}}$ or $\arity{v}=(c;c)$ for some $c\in C$;
\end{itemize}
Note that if $t$ does not belong to $\trees{C}'_S$ then $\hat(Z)(S)(t)=\varnothing$ and if $t\in \trees{C}'_S$ and there exists a morphism $\arr{t}{t'}$
then also $t'\in \trees{C}'_S$ so if $\hat{Z}'(S)$ is the restriction of $\hat{Z}(S)$ to $\opc{\trees{C}'}_S$ we have $\varinjlim \hat{Z}(S)=\varinjlim \hat{Z}'(S)$.

Now we want to add an inverse to every morphism $l$ in of $\trees{C}'_S$ of the following kinds:
\begin{enumerate}[(i)]
\item\label{mkind1} $\morp{l}{(T,\lambda_T,\mf{T})}{(T,\lambda_T,\mf{T}')}$ such that $T$ has a vertex $v$ such that $\arity{v}\in \Sqofc{A}$, $l$ is tree-preserving and $\mf{T}(v)=B$,  $\mf{T}'(v)=A$ and $\mf{T}(u)=\mf{T}'(u)$ for every $u\neq v$.
 \item\label{mkind2} $\morp{l}{(T,\lambda_T,\mf{T})}{(T,\lambda_T,\mf{T}')}$, where $T$  has a vertex $v$ such that $\arity{v}=(c;c)$ for $c\in C- \Cl{P}$, $l$ is the identity, $\mf{T}(v)=P$,  $\mf{T}'(v)=A$ and $\mf{T}(u)=\mf{T}'(u)$ for every $u\neq v$.
\end{enumerate}

The category $\trees{C}_S''$ is the one obtained by $\trees{C}_S'$ inverting all morphisms of kind (\ref{mkind1}) and (\ref{mkind2}). It exists, since
for every $T,T'\in \trees{C}_S''$ the set of morphisms $\trees{C}_S''(T,T')$ is a subset of the set of morphisms between the underlying trees.

Let $\morp{i}{\trees{C}_S'}{\trees{C}_S''}$ be the obvious inclusion.
All the morphisms of kind (\ref{mkind1}) and (\ref{mkind2}) are sent to isomorphisms by $\hat{Z}'$. For morphisms of kind (\ref{mkind1}) this follows from the fact that
if $S\in \Sqofc{\Cl{A}}$ then $\morp{g_!(i)_S}{g_! i_!A(S)}{g_! B(S)}$ is an isomorphism. For morphisms of kind (\ref{mkind2}) similarly this is a consequence of the fact that
if $c\in \Cl{Q}- \Cl{P}$ then $\morp{g_!(i)_{(c;c)}}{g_! i_!A(c;c)}{h_! P(c;c)}$ is an isomorphism. 

The functor $\hat{Z}'$ factors then through $i$  and a functor $\morp{\hat{Z}''}{\opc{\trees{C}_S''}}{\Mm}$, moreover $\varinjlim \hat{Z}'=\varinjlim \hat{Z}''$.

The object $S_P$ (the corolla associated to $S$ marked with $P$) is initial in $P\trees{C}_S''$.  To prove this fact, notice first that for every $t\in \trees{C}_S''$ there is at most one map from
$S_P$ to $t$; in fact the maps in $P\treesc{\Cl{Q}}''$ are uniquely determined by the underlying maps of $C$-trees, and for every $C$-tree $t$
with $\arity{t}=S$ and an order on the leaves $\tau$ there is a unique map from $\arity{S}$ to $t$ respecting $\tau$.

To prove that $S_P$ is initial in $P\trees{C}_S''$ is then sufficient to exhibit a map from $S_P$ to $t$ for every $t=(t,\mf{t},\lambda_t,\tau_t) \in P\trees{C}_S''$; this map is given
by the following composition: 

 \[
 S_P\overset{\theta}\longrightarrow \arity{t} \overset{q}\longrightarrow t'''\overset{r}\longrightarrow t''\overset{r}\longrightarrow t' \overset{s}\longrightarrow t,
\]

where  
\begin{itemize}
\item $t'=(t,\mf{t'},\lambda_t,\tau_t)$ and $s$ is a tree-preserving morphism such that for every vertex $v\in \verx{t}$:
\begin{itemize}
 \item if $\arity{v}\in \Sqofc{\Cl{A}}$ and $\mf{t}(v)=A$ then $\mf{t'}(v)=B$;
 \item if $\arity{v}\notin \Sqofc{\Cl{B}}$ then $\mf{t'}(v)=P$;
 \item otherwise $\mf{t'}(v)=\mf{t}(v)$.  
\end{itemize}  
The map $s$ is obtained using inverses of the maps of kind (\ref{mkind2}). 
Notice that $t'$ has no vertex marked by $A$.
\item $r$ is an iteration of marking-preserving inner faces of inner edges with both vertices marked by $B$. This process is carried out until $t''$ has no inner edges of this kind. As a consequence all the vertices $v$ in $t''$ that are marked by $B$ have arity in $\Sqofc{\Cl{A}}$;
\item $r$ is a tree preserving morphism such that on every vertex $v\in \verx{t'''}$ the marking in $P$. This map is obtained by compositions of inverses of morphism of kind \ref{mkind1};
\item $q$ is obtained by composition of marking-preserving inner faces;
\item $\theta$ is the unique (marking-preserving) isomorphism from $S_P$ to $\arity{t}$. 
\end{itemize}

Since $S_P$ is initial (final in the opposite category) the map 
\[
\morp{h_S}{h_!P(S)=Z(S)(S_P)}{Q(S)=\varinjlim Z(S)}
\]
 is an isomorphism, and this concludes the proof.
\end{proof}

We would like to remark that the statement of Proposition \ref{prop.ap.ff} remains true even if we do not suppose that $f$ is injective but in that case
the proof is a little bit more involved.

\section{Push-out alongs free maps and Topological Operads}\label{App.C}

The aim of this section is to prove that for every $C\in \Set$ the operad $\Opop{C}$ is $\mathcal{T}_1$-admissible in $\Top$, where $\Top$ the category of compactly generated (weak Hausdorff) spaces and $\mathcal{T}_1$ is the class of $T_1$ closed inclusions.

Recall that a \emph{$T_1$ closed inclusion} is a map $\morp{f}{X}{Y}$ in $\Top$ such that it is a closed inclusion and every point in $Y\backslash f(X)$ is closed.
The class of $T_1$ closed inclusions is monoidally saturated in $\Top$ and will be denoted by $\mathcal{T}_1$.

The monoidal model category $\Top$ admits a monoidal fibrant replacement functor and a contains a cocommutive comonoidal interval thus, thanks to Proposition \ref{path.admis} and \ref{BMstone}, it is sufficient to check condition \ref{item.adm1} of Definition \ref{Kadm} is satisfied by $\Opop{C}$. 
In other words we want to check that given a generating (trivial) cofibration $\morp{i}{K_0}{K_1}$ in $\Top^{\Sqofc{C}}$ and a map of $C$-coloured operads
$\morp{\alpha}{F_{\Opop{C}}(K_0)}{X}$ the map $i_{\alpha}$ in the push-out diagram
\begin{equation}\label{free.diag}
 \diagc{F_{\Opop{C}}(K_0)}{X}{F_{\Opop{C}}(K_1)}{Y}{F_{\Opop{C}}}{\alpha}{i_\alpha}{}{}
\end{equation}
is a local $\mathcal{T}_1$-morphism.

\newcommand{\TreeF}[1]{F\trees{#1}}
\newcommand{\TreeFm}[1]{F\underline{\mathbf{T}}(#1)}

Consider the category of ordered push-out $C$-trees $\trees{C}$ introduced in Section \ref{ss.pusht}, for simplicity this time the vertices can be marked by $K_0$ (instead of $o$), $K_1$ 
(instead of $y$) or $X$ (in place of $x$). 

Let us define $\TreeF{C}$, a subcategory of $\trees{C}$: it contains all objects of $\trees{C}$ and all morphisms $\morp{f}{T}{S}$ in $\trees{C}$ such
that:
\begin{itemize}
 \item $f$ decomposes as composition of inner-face maps and isomorphisms (no degeneracies are involved), i.e. if $v,v'$ are distinct vertices of $T$ then $f(v)\cap f(v')=\emptyset$;
 \item $f$ does not factor through inner-face maps associated to vertices marked by $K_0$ or $K_1$, i.e. for every vertex $v\in \verx{T}$ such that $M_T(v)\in \{K_0,K_1\}$ 
the set $f(v)$ is a singleton;
 \item the change of planar structure $\sigma_f\in G_T$ (\emph{cf.} Remark \ref{rmksigma}) is the identity on the vertex marked by $K_0$ or $K_1$.
\end{itemize}    

We will refer to $\TreeF{C}$ as the \emph{category of free-push-out $C$-trees}. 

For every $S\in \Sqofc{C}$ let us denote by $\TreeF{C}_S$ the full subcategory of $\TreeF{C}$ spanned by the objects $(T,\lambda_T,M_T,\tau_T)$ such that $\arity{T}=S$.
Note that $\TreeF{C}\simeq \underset{S\in \Sqofc{C}}\coprod \TreeF{C}_S$.
 
For every $S\in \Sqofc{C}$ we can define a functor
\[
 \morp{Z}{\opc{\TreeF{C}_S}}{\Mm}
\]
in the same way in which we define functor (\ref{tree.functor}):
\[
Z(T,M_T)\simeq \underset{v\in \verx{T}}\bigotimes M_T(v)(\arity{v}) 
\]

Using techniques similar to the ones used in Appendix \ref{App.B} one can check that
in diagram (\ref{free.diag})
\[
  Y(S)\simeq \varinjlim Z .
\]

Furthermore if we denote by $Z_0$ the restriction of $Z$ to the full subcategory of $\TreeF{C}_S$ spanned by the trees with vertices marked only by $X$ then
$X(S)\simeq \varinjlim Z_0$ and, under these identifications, $i_\alpha$ is the canonical map $\arr{\varinjlim Z_0}{\varinjlim Z}$.

The category $\TreeF{C}_S$ is quite big so it might be hard to prove that $i_\alpha$ is a $T_1$ closed inclusion. Since we know that $T_1$ closed inclusions are
closed under push-out and transfinite compositions, we will try to decompose this colimit as a transfinite composition of iterated push-outs along colimits of maps
which are simpler to understand.

Let us call a tree of $\TreeF{C}_S$ \emph{minimal} if none of its edges belongs to two vertices marked by $X$. The full subcategory of $\TreeF{C}_{S}$ spanned by minimal
trees is cofinal in $\TreeF{C}_{S}$ an will be denoted by $\TreeFm{C}_S$.  

For every $n\in \N$ we can define the following full subcategories of $\TreeFm{C}_S$:
\begin{itemize} 
\item $\TreeFm{C}_{S,\leq n}$ spanned by all the trees with at most $n$ vertices not marked by $X$;
\item $\TreeFm{C}_{S,n}$ spanned by all the trees with exactly $n$ vertices not marked by $X$;
\item $\TreeFm{C}_{S,n}^{0}$ spanned by all the trees with exactly $n$ vertices not marked by $X$ and at least one marked by $K_0$;
\item $\TreeFm{C}_{S,n}^{1}$ spanned by all the trees with exactly $n$ vertices not marked by $X$ and no vertices marked by $K_0$;
\item $\TreeFm{C}_{S,\leq n}^{+}=\TreeFm{C}_{S,\leq n}\cup\TreeFm{C}_{S,n+1}^{0}$
\end{itemize}
Let $Z_{\leq n}$, $Z_{n}$, $Z_{n}^{0}$, $Z_{n}^{1}$, $Z_{\leq n}^{+}$ be the restrictions of $Z$ to $\opc{\TreeFm{C}_{S,\leq n}}$, $\opc{\TreeFm{C}_{S,n}}$, $\TreeFm{C}_{S,n}^{0\,\mathrm{op}}$,
$\TreeFm{C}_{S,n}^{1\,\mathrm{op}}$ and $\TreeFm{C}_{S,\leq n}^{+\,\mathrm{op}}$ respectively. 

It it is clear that we have fully faithful inclusions $\morp{j_n}{\TreeFm{C}_{S,n}}{\TreeFm{C}_{S,n+1}}$ and that $\TreeFm{C}_S=\underset{n\in \N}\bigcup \TreeFm{C}_{S,n}$.
This produces a transfinite composition:
\begin{equation}\label{trans.comp.1}
 X(S)\simeq \varinjlim Z_{\leq 0} \overset{j_0} \longrightarrow \varinjlim Z_{\leq 1} \overset{j_1} \longrightarrow \dots \overset{j_{n-1}} \longrightarrow \varinjlim Z_{\leq n}
\overset{j_{n+1}} \longrightarrow \dots
\end{equation}
whose colimit is $\morp{i_\alpha(S)}{X(S)}{\varinjlim Z \simeq Y(S)}$. 

We would like to express this transfinite composition as an iteration of push-outs. 

We notice at first that $\TreeFm{C}_{S,\leq n}= \TreeFm{C}_{S,\leq n-1} \cup \TreeFm{C}_{S,n}=\TreeFm{C}_{S,n}^{1} \cup \TreeFm{C}_{S,\leq n-1}^{+}$.\\
Since there are no morphisms between objects in $\TreeFm{C}_{S,\leq n-1}$ and $\TreeFm{C}_{S,n}^{1}$ the following diagram is a push-out square:
\begin{equation}\label{it.po.1}
 \diagc{\varinjlim Z_{n}^{0}}{\varinjlim Z_{n}}{\varinjlim Z_{\leq,n-1}^{+}}{\varinjlim Z_{\leq,n}\ .}{}{u_n}{}{}{}
\end{equation}

Furthermore we have the following:

\begin{lemma}
For every $n\in \N$ the subcategory $\TreeFm{C}_{S,\leq n}$ is cofinal in $\TreeFm{C}_{S,\leq n}^{+}$ (or equivalently $\TreeFm{C}_{S,\leq n}^{\mathrm{op}}$ is final in $\TreeFm{C}_{S,\leq n}^{+\,\mathrm{op}}$).
\end{lemma} 

As a consequence diagram (\ref{it.po.1}) can be rewritten as:
\begin{equation}\label{it.po.2}
 \diagc{\varinjlim Z_{n}^{0}}{\varinjlim Z_{n}}{\varinjlim Z_{\leq,n-1}}{\varinjlim Z_{\leq,n}}{}{u_n}{}{}{}
\end{equation}

This means that the $n$-th map of transfinite composition (\ref{trans.comp.1}) can be expressed as a push-out along the canonical map:
\begin{equation}\label{canon.u}
\morp{u_n}{\varinjlim Z_{n}^0}{ \varinjlim Z_{n}}
\end{equation}

Note that $\TreeFm{C}_{S,n}^{0}$ and $\TreeFm{C}_{S,n}$ are fibred over $\TreeFm{C}_{S,n}^{1}$.

Indeed a functor $\morp{p}{\TreeFm{C}_{S,n}}{\TreeFm{C}_{S,n}^{1}}$ is defined as follows:
\begin{itemize}
\item for every object $(T, M_T)\in \TreeFm{C}_{S,n}$ the image $p((T,M_T))$ is $(T, M'_T)$, where for every $v\in \verx{T}$ the marking $M'_T(v)$ is $X$ if $M_T(v)=X$ and is $K_1$ otherwise.
\item for every morphism $\morp{f}{(T,M_T)}{(S,M_S)}$ the image $p(f)$ is $\theta_f$. 
\end{itemize}

It is easy to verify that this is a good definition and that $p$ is indeed a fibration or equivalently $\opc{p}$ is a opfibration.

For every $n\in \N$ we define the $n$-cube category $\square_n$ as the cartesian product $\Delta[1]^n$; the punctured $n$-cube category $\boxtimes_n$ is its full subcategory spanned by all
the objects different from the final object $(1,1,\dots,1)$. 

For every $T\in \TreeFm{C}_{S,n}^1$ the fiber $p^{-1}(T)$ is isomorphic to $\opc{\square_n}$ (which is actually isomorphic to $\square_n$).

The restriction $p_0$ of $p$ to $\TreeFm{C}_{S,n}^{0}$ is fibred over $ \TreeFm{C}_{S,n}^1$ and the fiber $p_0^{-1}(T)$ is isomorphic to $\opc{\boxtimes_n}$.

As a consequence the colimit of $Z_n^0$ and $Z_n$ can be calculated as the colimit indexed by $\TreeFm{C}_{S,n}^{1\, \mathrm{op}}$ of the colimits calculated on the fibers (\emph{cf.} \cite{He97}): 

\[
 \varinjlim Z_{n}\simeq \underset{\!\! T\in \TreeFm{C}_{S,n}^{1\,\mathrm{op}}}\varinjlim \underset{\, t\in p^{-1}(T)^{\mathrm{op}}} \varinjlim Z(t) 
\]
\[
 \varinjlim Z_{n}^{0}\simeq \underset{\!\! T\in \TreeFm{C}_{S,n}^{1\, \mathrm{op}}}\varinjlim \underset{t\in p_0^{-1}(T)^{\mathrm{op}}} \varinjlim Z(t) 
\]
thus we rewrite the map (\ref{canon.u}) as
\begin{equation}\label{canon.u2}
 u_n\simeq \underset{T\in \TreeFm{C}_{S,n}^{1\, \mathrm{op}}} \varinjlim u_{n,T} 
\end{equation}
where for every $T\in \TreeFm{C}_{S,n}^{1}$ the map $\morp{u_{n,T}}{\underset{t\in \opc{p_0^{-1}(T)}} \varinjlim Z(t)}{\underset{t\in \opc{p^{-1}(T)}} \varinjlim Z(t)}$ is the canonical one induced by the fully faithful
inclusion $\arr{p_0^{-1}(T)}{p^{-1}(T)}$.

\subsection{Description of $u_n,T$}

Let us fix an element $(T,M_T)\in \TreeFm{C}_{S,n}^{1}$ (here we explicit the marking map for convenience). Let $v_1,v_2,\dots, v_n \in \verx{T}$ be the $n$ vertices
 marked by $K_1$.

The element of $p^{-1}(T,M_T)$ are all $(S,M_S)\in \TreeFm{C}_{S,n}$ such that $S=T$ and for every $v\in \verx{T}$ one has $M_S(v)=X$ if and only if $M_T(v)=X$;
in other words all the $n$ vertices that are marked by $K_1$ in $T$ have to be marked by $K_0$ or $K_1$.
 
For every $a=(a_1\dots,a_n)\in \square_n$ let $\morp{M_{T,a}}{\verx{T}}{\{X,K_0,K_1\}}$ be the unique marking such that $M_{T,a}(v_m)=K_{a_m}$ for every $m\in [n]$ and $M_{T,a}(v)=X$ for every other vertex.  
Then there is an isomorphism of categories
\[
 \gfun{s_T}{\opc{\square_n}}{p^{-1}(T,M_T)}{a}{(T,M_{T,a})}
\]
which assigns to each morphism $\morp{g}{a}{b}$ in $\opc{\square_n}$ the unique tree-preserving morphism $\morp{s_T(g)}{(T,M_{T,a})}{(T,M_{T,b})}$.

Let $\morp{Z_T}{\square_n}{\Mm}$ be the composite $Z\circ \opc{s_T}$; then for every $b\in \square_n$
\[
 Z_T(b)\simeq (\underset{m\in [n]}\bigotimes K_{b_m}(\arity{v_m}))\otimes (\!\!\!\!\!\!\!\!\!\! \underset{v\in \verx{T}\backslash \{v_1,\dots,v_n\}} \bigotimes\!\!\!\!\!\!\!\!\!\! X(\arity{v})) 
\]
and for every map $\morp{l}{b}{c}$ in $\square_n$ the morphism $Z_T(l)$ is a tensor product of some $i(\arity{v_m})$'s and identities.

The category $\square_n$ has a final object $\bar{1}=(1,\dots,1)$ hence the codomain of $u_n$ is isomorphic to $Z_T(\bar{1})=Z(T,M_T)$. 
Furthermore 
\[
\morp{u_{n,T}}{\underset{b\in \boxtimes_n }\varinjlim Z_T(b)}{Z(T,M_T)} 
\]
can be calculated as an iterated pushout-product along the $i(\arity{v_m})$'s tensored with $\underset{v\in \verx{T}\backslash \{v_1,\dots,v_n\}} \bigotimes\!\!\!\!\!\!\!\!\!\! X(\arity{v})$ by a well known argument.

This implies in particular that:
\begin{lemma}\label{mon.sat.cube}
 For every $T\in \TreeFm{C}_{S,n}^{1}$ and $n\in\N$ the morphism $u_{n,T}$ belongs to the monoidally saturated class of the set of maps $\{i(S)\ |\ S\in \Sqofc{C}\}$.
\end{lemma}

\subsection{$u_n$ as quotient of equivariant maps}
The category $\TreeFm{C}_{S,n}^{1}$ is a groupoid, in facts all its morphisms are marking-preserving isomorphisms of trees. Furthermore each Hom-set is finite, thus each
connected component of $\TreeFm{C}_{S,n}^{1}$ is equivalent to a finite group. Let $\mathcal{G}_S$ be the set of connected components of $\TreeFm{C}_{S,n}^{1}$ and
suppose that a representative $T_G\in G$ is chosen for every $G\in \mathcal{G}_S$.

The morphism $u_n$ (formula \ref{canon.u2}) is then isomorphic to the coproduct of maps between orbits:
\begin{equation}
 u_n=\underset{G\in\mathcal{G}_S}\coprod (u_{n,T_G}/ G)
\end{equation}

\begin{rmk}
 Every path component of $G\in \TreeFm{C}_{S,n}^{1}$ is equivalent to the group of automorphisms of its representative $T_G$. Suppose $T_G=(T,\lambda_T,M_T,\tau_T)$,
then an automorphism of $T_G$ is just a (non-planar) automorphism of $T$ (a $C$-tree) preserving the order on the leaves $\tau_T$, the marking $M_T$ (which takes values in $\{X,K_1\}$) and whose change of planar structure is trivial in
vertices marked by $K_1$. Note that if $T$ does not have vertices of arity $0$ then there are no non-trivial automorphism of $T_G$, since in this case the group
of automorphism of $T$ acts freely on the set of linear orders on the leaves of $T$. 
\end{rmk}

We remark that our decomposition of $i_\alpha$ in a transfinite composition of push-outs is similar to the one obtained in \cite[Sections 5.8-5.11]{BM03} where the push-out along a free map generated by a map of collections is considered.

\subsection{$\mathcal{T}_1$ admissibility of $\Opop{C}$ in $\Top$}

We can now prove that $\Opop{C}$ is $\mathcal{T}_1$ is admissible in $\Top$. Thanks to Proposition \ref{path.admis} and Proposition \ref{BMstone} we are reduced to prove the following:

\begin{lemma}
 Suppose a diagram as (\ref{free.diag}) is given in $\EOperfc{\Top}{C}$. If $i$ is a local cofibration then $i_\alpha$ is a local $\mathcal{T}_1$-morphism.
\end{lemma}
\begin{proof}
For every $S\in \Sqofc{C}$ the map $i_{\alpha}(S)$ can be expressed as the transfinite composite of (\ref{trans.comp.1}). For every $n\in\N$ the map $j_n$ is a
push-out of $u_n$ (\emph{cf.} (\ref{it.po.1}). Whereas the class of $T_1$ closed inclusions is saturated it is sufficient to show that $u_n$ belongs to this class
for every $n\in \N$. From Lemma \ref{mon.sat.cube} we deduce that $u_{n,T}$ is a $T_1$ closed inclusion for every $T\in \TreeFm{C}_{S,n}^{1}$, since the class of 
$T_1$ closed inclusions is monoidally saturated. 
We can conclude that $u_n$ is a $T_1$ closed inclusion by applying Lemma \ref{T1.G}.
\end{proof}

\begin{lemma}\label{T1.G}
Let $G$ be a finite discrete group. The quotient of a $G$-map in $\Top$ whose underlying map is in $\mathcal{T}_1$ is in $\mathcal{T}_1$.  
\end{lemma}
\begin{proof}
\[
 \diagc{X}{X/G}{Y}{Y/G}{f}{\pi_X}{f/G}{\pi_Y}{}
\]
Note that since $G$ is finite $\pi_X$ and $\pi_Y$ are both closed.
It is easy to check that $f/G$ is a closed inclusion. 
If $[y]\in Y/G$ belongs to $(Y/G)\backslash (X/G)$ then $y$ belongs to $Y\backslash X$ hence it is closed. Since $\pi_Y$ is a closed map we conclude that $[y]$ is closed as well.
  
\end{proof}

\subsection{Tameness}
We would like to stress that the category $\TreeFm{C}^{1\, \mathrm{op}}$ is final in the category $\mathbf{T}^{T+1}$ defined in \cite[Section 7]{BB13} for the (polynomial) monad $T$ associated to $\Opop{C}$. 
More in general suppose we have a monoidal bicomplete category $\Mm$, a polynomial monad $T$ with set of colours $C$  and a class of morphism $A$ in $\Mm$ monoidally saturated and closed respect to colimits indexed by $\mathbf{T}^{T+1}$.
Suppose that $i$ is a local $A$-morphism in $\Mm^{C}$. It can be proven that in a push-out diagram in the category $\Alg{T}{\Mm}$ along the free map $F_{T}(i)$ like (\ref{free.diag}) the map $i_\alpha$ is a local $A$-morphism.

By definition a polynomial monad is tame if $\mathbf{T}^{T+1}$ has a final object in every connected component so the condition on $A$ to be closed respect to colimits over $\mathbf{T}^{T+1}$ becomes empty.
In the case of $\Opop{C}$ we have seen that in $\mathbf{T}^{\Opop{C}+1}$ each component has a final subcategory which is equivalent to a group.       
\bibliographystyle{plain}
\bibliography{eopbib}
\end{document}